\documentclass[11pt]{article}
\usepackage[T2A]{fontenc}
\usepackage[utf8]{inputenc}
\usepackage{fullpage}
\usepackage{amsmath,amsfonts,amssymb,amsthm,enumerate,mdframed,hyperref,cleveref,algorithmic,algorithm,multirow,import,tikz,booktabs,xcolor,soul,mathtools}
\usetikzlibrary{arrows,positioning, shapes} 
\usepackage{caption,subcaption}
\usepackage{authblk}

\newtheorem{theorem}{Theorem}
\newtheorem{proposition}[theorem]{Proposition}
\newtheorem{lemma}[theorem]{Lemma}

\newtheorem{conjecture}[theorem]{Conjecture}
\newtheorem{definition}[theorem]{Definition}
\newtheorem{corollary}[theorem]{Corollary}

\mathcode`*=\numexpr\mathcode`*-"2000\relax

\definecolor{darkgreen}{rgb}{0,0.6,0}

\providecommand{\F}{\mathbb{F}}
\providecommand{\N}{\mathbb{N}}

\DeclareMathOperator{\AG}{AG}
\DeclareMathOperator{\PG}{PG}
\newcommand{\<}{\langle}
\renewcommand{\>}{\rangle} % was: tabbing command
\newcommand{\cA}{\mathcal{A}}
\newcommand{\cC}{\mathcal{C}}
\newcommand{\cP}{\mathcal{P}}
\newcommand{\cQ}{\mathcal{Q}}
\newcommand{\cS}{\mathcal{S}}
\newcommand{\cT}{\mathcal{T}}
\newcommand{\cU}{\mathcal{U}}
\newcommand{\gaussm}[3]{\genfrac{[}{]}{0pt}{}{#1}{#2}_{#3}}

\newcommand{\avsp}{avsp}
\newcommand{\coloneq}{\vcentcolon=}      %centered colon for := uses mathtools
\newcommand{\eqcolon}{=\vcentcolon}      %centered colon for =: uses mathtools
\providecommand{\keywords}[1]
{
  \small	
  \textbf{\textit{Keywords---}} #1
}
\newcommand{\rev}[1]{#1}

\title{Affine vector space partitions}
\author[1]{John Bamberg}
\author[2]{Yuval Filmus}
\author[3]{Ferdinand Ihringer}
\author[4]{Sascha Kurz}
\affil[1]{The University of Western Australia, Australia, john.bamberg@uwa.edu.au}
\affil[2]{The Henry and Marilyn Taub Faculty of Computer Science and Faculty of Mathematics, Technion --- Israel Institute of Technology, Israel, yuvalfi@cs.technion.ac.il}
\affil[3]{Ghent University, Belgium, Ferdinand.Ihringer@gmail.com}
\affil[4]{University of Bayreuth, Germany, sascha.kurz@uni-bayreuth.de}
\date{May 2023}

\begin{document}

\maketitle

\begin{abstract}
  An affine vector space partition of $\AG(n,q)$ is a set of proper affine subspaces that partitions the set of points. Here we determine minimum sizes and enumerate equivalence classes of affine vector space partitions for small parameters. We also give parametric constructions for arbitrary field sizes.
\end{abstract}
%% please replace if you find more meaningful keywords

\section{Introduction}
\label{sec_introduction}

A \emph{vector space partition} $\cP$ of the projective space $\PG(n-1,q)$ is a set of subspaces in $\PG(n-1,q)$ which partitions the set of points. For a survey on known results we 
refer to \cite{heden2012survey}. We say that a vector space partition $\cP$ has type $(n-1)^{m_{n-1}} \dots 2^{m_2}1^{m_1}$ if precisely $m_i$ of its elements have dimension $i$, where $1\le i\le n$. The 
classification of the possible types of a vector space partition, given the parameters $n$ and $q$, is an important and difficult problem. Based on \cite{heden1986partitions}, the classification
for the binary case $q=2$ was completed for $n\le 7$ in \cite{el2009partitions}. Under the assumption $m_1=0$ the case $(q,n)=(2,8)$ has been treated in \cite{el2010partitions}. It seems quite 
natural to define a \emph{vector space partition} $\cA$ of the affine space $\AG(n,q)$ as a set of subspaces in $\AG(n,q)$ that partitions the set of points. However, it turns out that 
those partitions exist for all types which satisfy a very natural numerical condition. If we impose the additional condition of \emph{tightness}, that is that the projective closures of the elements of $\cA$ have an empty intersection, then
the classification problem becomes interesting and challenging. This condition is natural in the context of hitting formulas as introduced in \cite{Iwama}, that is for 
logical formulas in full disjunctive normal form (DNF) such that each truth assignment to the underlying variables satisfies precisely one term. For a more recent treatment and applications we refer to \cite{PeitlSzeider}. Here we consider the geometrical and the combinatorial point of view. 
%% As results we obtain a complete classification of all possible types of vector space partitions in $\AG(n,2)$ for $n\le 7$. 

Variants of vector space partitions of $\PG(n-1,q)$ have been studied in the literature. In \cite{el2011lambda} the authors study (multi-)sets of subspaces
covering each point exactly $\lambda$ times. The problem of covering each $k$-space exactly once is considered in \cite{heinlein2019generalized}. A more general partition problem 
for groups is studied in \cite{heden1986partitions}. 
Irreducible homogeneous affine vector space partitions have been 
studied by Agievich \cite{Agievich08} and Tarannikov \cite{Tarannikov22} motivated by the study of bent functions.
However, we are not aware of any publication treating the introduced affine vector space partitions
in the same generality as in the present work.

The paper is organized as follows. In Section~\ref{sec_preliminaries} we formally introduce affine vector space partitions, state the preliminaries, and develop
the first necessary existence conditions. Here we are guided by the published necessary conditions for vector space partitions. 
We also argue why tightness (see above) and \emph{irreducibility}, that is there exists no proper subset $\cA' \subsetneq \cA$ such that 
the union of all elements of $\cA'$ is a subspace of $\AG(n,q)$, are necessary to obtain an interesting existence question.
% that we need 
% the two additional conditions, tightness and irreducibility, to make the existence question interesting. 
In Section~\ref{sec_classification_avsp} we classify 
affine vector space partitions for arbitrary field sizes but small dimensions. Section~\ref{sec_classification_ti_avsp} is concerned with the binary case. We completely determine 
the possible dimension distributions of tight irreducible affine vector space partitions of $\PG(n-1,2)$ for all $n\le 7$. In a few cases we give theoretical or computational 
classifications of the corresponding equivalence classes of tight irreducible vector space partitions. A very nice example consists of eight solids in $\PG(6,2)$ whose parts at infinity 
live on the Klein quadric $Q^+(5,2)$. A generalization to arbitrary finite fields of characteristic $2$ is given in Subsection~\ref{subsec_klein_quadric}. Parametric constructions of tight irreducible affine vector space partitions using spreads or hitting formulas complete Section~\ref{sec_constructions}.  
% \ref{sec_constructions}, which also contains several other parametric constructions 
%of tight irreducible affine vector space partitions. A construction using hitting formulas is briefly touched in Subsection~\ref{subsec_hitting_formulas}. 
In Section~\ref{sec_minimum_size} 
we determine the smallest possible size of an irreducible tight affine vector space partition of $\PG(7,2)$ and give a parametric upper bound for $\PG(n-1,2)$ of size 
roughly $\tfrac{3n}{2}$, which is significantly smaller than the conjectured smallest size of an irreducible hitting formula mentioning all variables. We close with a conclusion and a 
list of open problems in Section~\ref{sec_conclusion}. To keep the paper self-contained we present some additional material in an appendix. Section~\ref{sec_ilp} contains details on 
integer linear programming formulations that we have utilized to obtain some computational results. Section~\ref{sec_details} contains a few technical results that might be left to the 
reader or collected from the literature. Lists of hitting formulas that can be used to construct tight irreducible affine vector space partitions of the minimum possible size are
given in Section~\ref{sec_compression}.

\section{Preliminaries and necessary conditions}
\label{sec_preliminaries}
\begin{definition} 
  An \emph{affine vector space partition} $\cA$ of $\AG(n,q)$ is a set $\left\{A_1,\dots,A_r\right\}$ of subspaces of $\AG(n,q)$ such that $0\le \dim\!\left(A_i\right)\le n-1$ for all 
  $1\le i\le r$ and every point (element of \rev{$\F_q^{n}$}) 
  %% $\F_q^{n}\backslash\mathbf{0}$) 
  is contained in exactly one element $A_i$. The integer $r$ is called the \emph{size} of the 
  affine vector space partition.
\end{definition}
% The assumption $\dim\!\left(A_i\right)\neq n$ excludes the possibility of choosing the set of all points as unique element of the partition. 
We write $\#\cA$ for the size of $\cA$. 
For each affine subspace $A\in\AG(n,q)$ we write $\overline{A}$ for its projective closure. With this $\overline{\cA}\coloneq\left\{\overline{A}\,:\,A\in\cA\right\}$ is the natural embedding
of an affine vector space partition of $\AG(n,q)$ in $\PG(n,q)$. Denoting the hyperplane at infinity by $H_\infty$, we can directly define an affine vector space partition in 
$\PG(n,q)$:
\begin{definition}
  An \emph{affine vector space partition} $\cU$ of $\PG(n-1,q)$ is a set $\left\{U_1,\dots,U_r\right\}$ of subspaces of $\PG(n-1,q)$ such that $1\le \dim\!\left(U_i\right)\le n-1$ for all 
  $1\le i\le r$ and there exists a hyperplane $H_\infty$ such that every point ($1$-dimensional subspace) outside of $H_\infty$ is contained in exactly one element $U_i$ 
  and $U_i\not\le H_\infty$ for all $1\le i\le r$. The integer $r$ is called the \emph{size} of the affine vector space partition and also denoted by $\#\cU$.
\end{definition}
Here we use the algebraic dimension for subspaces in $\PG(n-1,q)$, i.e., if $\dim(U)=u$, then $\# U=\gaussm{u}{1}{q}\coloneq\tfrac{q^u-1}{q-1}$ and we also speak of $u$-spaces. Using 
the geometric language, we call $1$-, $2$-, $3$-, $4$-, and $n{-}1$-spaces points, lines, planes, solids, and hyperplanes, respectively. For each $1\le i\le r$ the set 
$U_i\backslash H_\infty$ is an affine space containing \rev{$q^{\dim(U_i)-1}$ points}.

In the following we will mostly speak of an affine vector space partition, abbreviated as \avsp, and will consider its embedding in $\PG(n-1,q)$. The \emph{type} of an {\avsp} 
$\cU=\left\{U_1,\dots,U_r\right\}$ is given by $(n-1)^{m_{n-1}}\dots 2^{m_2}1^{m_1}$, where $m_i=\#\left\{U_j\,:\, 1\le j\le r, \dim(U_j)=i\right\}$. Counting points 
outside of $H_\infty$ gives
\begin{equation}
  \sum_{i=1}^{n-1} m_i \cdot q^{i-1}=q^{n-1}. \label{eq_packing_condition}
\end{equation}
The analog of Equation~(\ref{eq_packing_condition}) for vector space partitions of $\PG(n-1,q)$ is called the \emph{packing condition}. While the packing condition for vector 
space partitions of $\PG(n-1,q)$ is just a necessary but not a sufficient condition for the existence with a given type, for 
{\avsp}s Equation~(\ref{eq_packing_condition}) is both necessary and sufficient.                                  
\begin{lemma}
  \label{lemma_all_types_are_feasible}
  For each type $(n-1)^{m_{n-1}}\dots 2^{m_2}1^{m_1}$ that satisfies the packing condition (\ref{eq_packing_condition}) there exists an {\avsp} $\cU$ of $\PG(n-1,q)$ attaining that type.
\end{lemma}                                
\begin{proof}
  Consider a subspace $K$ of $H_\infty$ with $\dim(K)=n-2$. %%, i.e., a subspace of co-dimension $2$ contained in $H_\infty$. 
  By $H_1,\dots, H_q$ we denote the $q$ hyperplanes 
  containing $K$ that are not equal to $H_\infty$. Clearly, we have $0\le m_{n-1}\le q$ and we can choose $H_1,\dots,H_{m_{n-1}}$ as the first elements of $\cU$. The remaining 
  elements are constructed recursively. For each index $m_{n-1}+1\le j\le q$ we consider an {\avsp} of type 
  $(n-2)^{m_{n-2}^{(j)}}\dots 2^{m_2^{(j)}}1^{m_1^{(j)}}$ where the $m_i^{(j)}\in\N_0$ are chosen such that the packing condition is satisfied for $H_j$ and 
  \begin{equation}
    \sum_{j=m_{n-1}+1}^q m_i^{(j)}=m_i 
  \end{equation}
  for all $1\le i\le n-2$. Such a decomposition can be easily constructed, see e.g.\ Algorithm~\ref{algo_decomposition_packing_formula} in Section~\ref{sec_details}.
\end{proof}

\begin{definition}
  We call an {\avsp} $\cU=\left\{U_1,\dots,U_r\right\}$ \emph{reducible} if there exists a subspace $U$ and a subset $S\subsetneq \{1,\dots,r\}$ 
  such that $\dim(U)<n$, $\#S>1$ and $\left\{U_i\,:\, i\in S\right\}$ is an {\avsp} of \rev{$U$}. Otherwise $\cU$ is called \emph{irreducible}. 
\end{definition}

%% For $\alpha,\beta\in\PG(n-1,q)$ we say that We say that $\beta$ \emph{escapes} $\alpha$ if $\alpha\cap \beta\neq \emptyset$ and $\alpha \not\le \beta$.  
%% \renewcommand{\algorithmicrequire}{\textbf{Input:}}
%% \renewcommand{\algorithmicensure}{\textbf{Output:}}
%% \begin{algorithm} 
%% \begin{algorithmic}
%% \REQUIRE Affine vector space partition $\cU=\left\{U_1,\dots,U_r\right\}$ of $\PG(n-1,q)$
%% \ENSURE Whether $\cU$ is reducible
%% \FOR{$1 \leq i < j \leq r$}
%%     \STATE $\alpha \gets \left\langle U_i,U_j\right\rangle$
%%     \STATE $S\gets \{i,j\}$
%%     \WHILE{some $U_k$ escapes $\alpha$}
%%         \STATE $\alpha \gets \left\langle \alpha ,U_k\right\rangle$
%%         \STATE $S\gets S\cup\{k\}$
%%     \ENDWHILE
%%     \IF{$\alpha \neq \PG(n-1,q)$}
%%         \RETURN Reducible: $\left\{U_h\,:\, h\in S\right\}$ is an affine vector space partition of $\alpha$
%%     \ENDIF
%% \ENDFOR
%% \RETURN Irreducible
%% \end{algorithmic}
%% \caption{Algorithm for checking whether an affine vector space partition is reducible}
%% \label{alg:reducibility}
%% \end{algorithm}
%% 
%% \begin{theorem} \label{thm:reducibility-alg}
%% \Cref{alg:reducibility} correctly determines whether the input affine vector space partition is reducible and runs in polynomial time.
%% \end{theorem}
%% \begin{center}
%%   \textbf{ToDo:} Add proof.
%% \end{center}

\begin{lemma}
  The smallest size of an irreducible {\avsp} $\cU$ of $\PG(n-1,q)$ is given by $\#\cU=q$.
\end{lemma}
\begin{proof}
  Let $\cU$ be an {\avsp} of $\PG(n-1,q)$. Since there are $q^{n-1}$ points to cover and each subspace covers at most $q^{n-2}$ points, 
  we have $\#\cU\ge q$. Now consider a hyperplane $K$ of $H_\infty$. By $H_1,\dots,H_q$ we denote the $q$ hyperplanes containing $K$ and not being equal to $H_{\infty}$. 
  With this, $\left\{H_1,\dots,H_q\right\}$ is an irreducible {\avsp} of $\PG(n-1,q)$.
\end{proof}                            

For a vector space partition $\cP$ of $\PG(n-1,q)$ we have $\dim(A)+\dim(B)\le n$ for each pair $\{A,B\}$ of different elements of $\cP$, which is also called \emph{dimension 
condition}. Using this it can be easily shown that $\#\cP\ge \tfrac{q^n-1}{q^{n/2}-1}=q^{n/2}+1$ if $n$ is even and $\#\cP\ge q^{(n+1)/2} +1$ if $n$ is odd. Both bounds can be attained by 
spreads, i.e., vector space partitions of type $(n/2)^{q^{n/2}+1}$, and lifted MRD codes of maximum possible rank distance, i.e., vector space partitions of type 
$((n+1)/2)^1 ((n-1)/2)^{q^{(n+1)/2}}$, respectively. In \cite{nuastase2011minimum} the authors determine the minimum size $\sigma_q(n,t)$ of a vector space partition of $\PG(n,q)$ 
whose largest subspace has dimension $t$.

\begin{lemma}
  \label{lemma_irreducible_no_multiset}
  Let $\cU$ be an irreducible {\avsp} of $\PG(n-1,q)$ and $U_1,\dots,U_q\in\cU$ be $q$ different elements with 
  $\dim(U_1)=\dots=\dim(U_q)$ and $\dim\left(\left\langle U_1,\dots, U_q\right\rangle\right)=\dim(U_1)+1$. Then we have $\dim(U_1)=\dots=\dim(U_q)=n-1$.
\end{lemma}
\begin{proof}
  Let $U\coloneq\left\langle U_1,\dots, U_q\right\rangle$ and $u\coloneq\dim(U_1)$. Since $U\backslash H_\infty$ contains $q^u$ points and $U_i\backslash H_\infty$ contains 
  $q^{u-1}$ points for each $1\le i\le q$, the set $\cU\backslash\left\{U_1,\dots,U_q\right\}\cup\{U\}$ is an {\avsp} unless $\dim(U)=u+1=n$. 
\end{proof}

\begin{corollary}
  Let $\cU$ be an irreducible {\avsp} of $\PG(n-1,2)$ and $U_1,U_2\in\cU$ be two different elements with $\dim(U_1)=\dim(U_2)=\dim(U_1\cap U_2)+1$. 
  Then, we have $\dim(U_1)=\dim(U_2)=n-1$.     
\end{corollary}

\noindent 
As an %%, relatively weak, 
analog of the dimension condition for vector space partitions in $\PG(n-1,q)$ we have:
\begin{lemma}
  \label{lemma_dimension_condition}
  Let $\cU$ be an {\avsp} in $\cU$. For each $U,U'\in \cU$ we have
  \begin{equation}
    \dim(U\cap U')=\dim(U\cap U'\cap H_\infty)\ge \dim(U)+\dim(U')-n.
  \end{equation}
\end{lemma}
\begin{proof}
  Since $U\backslash H_\infty$, $U'\backslash H_\infty$ are disjoint and $U,U'\not\le H_\infty$ we have $\dim(U\cap U')=\dim(U\cap U'\cap H_\infty)$. The inequality 
  follows from $\dim(U_1\rev{\cap} U_2)+\dim(\langle U_1,U_2\rangle)=\dim(U_1)+\dim(U_2)$ and $\dim(\langle U_1,U_2\rangle)\le n$.  
\end{proof}

\noindent
Due to the following general construction for (irreducible) {\avsp}s we introduce a further condition. 
\begin{lemma}
  Let $\cU=\left\{U_1,\dots,U_r\right\}$ be an {\avsp} of $\PG(n-1,q)\eqcolon V$ and $P$ be a point outside of $V$ (\rev{embedded in $\PG(n,q)$}). Then, $\cU'\coloneq\left\{\left\langle U_1,P\right\rangle,\dots,\left\langle U_r,P\right\rangle\right\}$ is an {\avsp} 
  of $\langle V,P\rangle\cong \PG((n+1)-1,q)$, \rev{where the {\lq\lq}new{\rq\rq} hyperplane at infinity arises via $\langle H_\infty,P\rangle$ from the {\lq\lq}old{\rq\rq} hyperplane at infinity $H_\infty$}. Reducability  inherits, i.e.\ $\cU'$ is irreducible iff $\cU$ is irreducible.
\end{lemma}
%% \begin{lemma}
%%   Let $\cU=\left\{U_1,\dots,U_r\right\}$ be an {\avsp} of $\PG(n-1,q)\eqcolon V$ and $P$ be a point outside of $V$ (embedded in $\PG(n'-1,q)$ for 
%%  some larger value of $n'$). Then, $\cU'\coloneq\left\{\left\langle U_1,P\right\rangle,\dots,\left\langle U_r,P\right\rangle\right\}$ is an {\avsp} 
%%   of $\langle V,P\rangle\cong \PG((n+1)-1,q)$. Reducability  inherits, i.e.\ $\cU'$ is irreducible iff $\cU$ is irreducible.
%% \end{lemma}

\begin{definition}
  Let $\cU=\left\{U_1,\dots,U_r\right\}$ be an {\avsp} of $\PG(n-1,q)$. We call $\cU$ \emph {tight} iff 
  the intersection of all $U_i$ does not contain a point, i.e.\ the intersection is trivial.
\end{definition}   
The same definition was proposed by Agievich~\cite{Agievich08}, under the name \emph{primitivity}, and was dubbed \emph{A-primitivity} by Tarannikov~\cite{Tarannikov22}.
We remark that an {\avsp} $\cA$ of $\AG(n,q)$ is tight iff for any $x\in\F_q^{n}$, there exists an $A\in\cA$ such that $A$ is not invariant under addition of $x$, that is $A+x \neq A$. 
%% Checking tightness of an affine vector space partition $\cA$ in $\PG(n-1,q)$ can be done in polynomial time since $U\cap U'$ can be computed using row reduced echelon forms and Gaussian elimination.            

\begin{lemma}
  \label{lemma_lb_tight_avsp}
  For each integer $n\ge 2$ there exists a tight {\avsp} of $\PG(n-1,q)$ with size $(q-1)\cdot(n-1) +1$. 
\end{lemma}
\begin{proof}
  Apply the following recursive construction. Start with an $(n-2)$-dimensional subspace $K$ of $H_\infty$ and consider the $q$ hyperplanes 
  $H_1,\dots,H_q$ containing $K$ but not being equal to $H_\infty$. Choose $q-1$ out of these and continue the iteration with the remaining 
  hyperplane until it becomes $2$-dimensional, i.e.\ a line. In the final step replace the affine line by $q$ points, so that the resulting 
  {\avsp} is trivially tight.
\end{proof}
A classical result in computer science, attributed to Tarsi, states that a minimally unsatisfiable \rev{Boolean formula in conjunctive normal form (CNF)} with $m$ clauses mentions at most $m-1$ variables, see e.g.\ 
\cite[Theorem 8]{DDK98}. The proof can be slightly modified to show that for $n\ge 2$ each tight {\avsp} of $\PG(n-1,2)$ has size at least $n$. \rev{One might conjecture that Lemma~\ref{lemma_lb_tight_avsp} is tight. Some preliminary results in that direction are proven in \cite[Sec. 3.2]{filmus2022irreducible}.}
%% We will prove the conjecture that Lemma~\ref{lemma_lb_tight_avsp} is tight in a subsequent paper. 
The determination of the minimum size of an irreducible tight {\avsp} is 
quite a challenge and we will present our preliminary results in Sections~\ref{sec_classification_avsp} and \ref{sec_classification_ti_avsp}. 

Note that tightness and irreducibility can be checked efficiently.
In particular, for irreducibility it suffices to calculate 
the affine closure for all pairs of subspaces in the avsp.
We will show efficiency formally and provide detailed algorithms in future work on hitting formulas.

% We remark that the properties {\lq\lq}tight{\rq\rq} and {\lq\lq}irreducible{\rq\rq} can be checked efficiently. Details on algorithmic issues are postponed to future work on hitting formulas.

\begin{lemma}
  \label{lemma_characterization_full}
  Let $U$, $K$, and $H_\infty$ be subspaces in $\PG(n-1,q)$ with $K\le H_\infty$, $\dim(K)=n-2$, $\dim(H_\infty)=n-1$, and $\dim(U\cap H_\infty)=\dim(U)-1$, i.e.\ 
  $U\not\le H_\infty$. By $H_1,\dots,H_q$ we denote the $q$ hyperplanes containing $K$ but not being equal to $H_\infty$. Then the following statements are 
  equivalent:
  \begin{itemize}
    \item[(1)] $U\cap H_\infty\le K$;\\[-6mm]
    \item[(2)] there exists an index $1\le i\le q$ with $U\le H_i$;\\[-6mm]
    \item[(3)] there exists an index $1\le i\le q$ with $U\le H_i$ and $U\cap H_j=U\cap H_\infty=U\cap K$ for all $1\le j\le q$ with $j\neq i$;\\[-6mm]
    \item[(4)] $\dim(U\cap K)=\dim(U)-1$.
  \end{itemize}    
\end{lemma}

\begin{lemma}
  \label{lemma_characterization_splitted}
  Let $U$, $K$, and $H_\infty$ be subspaces in $\PG(n-1,q)$ with $K\le H_\infty$, $\dim(K)=n-2$, $\dim(H_\infty)=n-1$, and $\dim(U\cap H_\infty)=\dim(U)-1$, i.e.\ 
  $U\not\le H_\infty$. By $H_1,\dots,H_q$ we denote the $q$ hyperplanes containing $K$ but not being equal to $H_\infty$. Then the following statements are 
  equivalent:
  \begin{itemize}
    \item[(1)] $U\cap H_\infty\not\le K$;\\[-6mm]
    \item[(2)] $\dim(U\cap H_i)=\dim(U)\rev{-1}$ for all $1\le i\le q$;\\[-6mm]
    \item[(3)] there are $q$ $(\dim(U)-1)$-spaces in $U$ containing $U\cap K$ and not being contained in $H_\infty$;\\[-6mm]
    \item[(4)] $\dim(U\cap K)=\dim(U)-2$.
  \end{itemize}    
\end{lemma}

Assume that $\cP$ is a vector space partition of $\PG(n-1,q)$ with type  $k_1^{m_1} \dots k_l^{m_l}$, where $k_1 > \dots > k_l$ and $\rev{m}_i>0$ for all $1\le i\le l$. The so-called 
\emph{tail} $\cT$ of $\cP$ is the set of all $k_l$-spaces in $\cP$, i.e., the set of all elements with the smallest occurring dimension. In \cite{heden2009length} several
conditions on $\#\cT$ have been obtained. In our situation we can also consider the \emph{tail} $\cT\coloneq\left\{U\in\cU\,:\,\dim(U)=k_l\right\}$ of an {\avsp } 
of $\PG(n-1,q)$ with type $k_1^{m_1} \dots k_l^{m_l}$, where $k_1 > \dots > k_l$ and $\rev{m}_i>0$ for all $1\le i\le l$. The packing condition~(\ref{eq_packing_condition}) directly 
implies that $q^{k_{l-1}-k_{l}}$ divides $\#\cT=m_l$ if $l\ge 2$ and that \rev{$q^{n-k_l}$} divides $\#\cT=m_l$ if $l=1$. In \cite{kurz2018heden} the results on the tail of a vector space partition 
of $\PG(n-1,q)$ were refined using the notion of $\Delta$-divisible sets of $k$-spaces.

\begin{definition}
  A (multi-)set $\cS$ of $k$-spaces in $\PG(n-1,q)$ is called $\Delta$-divisible iff $\#\cS\equiv \#(H\cap \cS)\pmod\Delta$ for every hyperplane $H$, where $H\cap \cS$ 
  denotes the (multi-)set of elements of $\cS$ that are contained in $H$.  
\end{definition}

\rev{We use the notation $\{\{\star\}\}$ for a multiset $\cS$ and $\#\cS$ for its cardinality.}

\begin{lemma}
  \label{lemma_divisible_sets_of_k_spaces}
  Let $\cU$ be an {\avsp} of $\PG(n-1,q)$ of type $k_1^{m_1}\dots k_l^{m_l}$, where $k_1>\dots>k_l>1$ and $\rev{m}_i>0$ for all $1\le i\le l$. Let 
  $\cT\coloneq\left\{U\in\cU\,:\, \dim(U)=k_l\right\}$ be the tail of $\cU$ and $\cT'\coloneq\left\{ T\cap H_\infty\,:\, T\in\cT \right\}$. If $l\ge 2$, then 
  $\cT'$ is $q^{k_{l-1}-k_l}$-divisible and $\#\cT=\#\cT'\equiv 0\pmod {q^{k_{l-1}-k_l}}$. If $l=1$, then $\cT'$ is $q$-divisible and $\#\cT=\#\cT'\equiv 0\pmod {\rev{q^{n-k_l}}}$. 
\end{lemma}
\begin{proof}
  Clearly we have $\#\cT=\#\cT'$. 
  From the packing condition (\ref{eq_packing_condition}) we directly conclude $\#\cT\equiv 0\pmod {q^{k_{l-1}-k_l}}$ if $l\ge 2$ and $\#\cT\rev{=q^{n-k_l}}\equiv 0\pmod {\rev{q^{n-k_l}}}$ if $l=1$. Let 
  $K$ be an arbitrary hyperplane of $H_\infty$ and $H_1,\dots,H_q$ be the $q$ hyperplanes of $\PG(n-1,q)$ not being equal to $H_\infty$. Call the points outside of $H_\infty$ that 
  are contained in some element of $\cU$ with dimension strictly larger than $k_l$ covered and all others outside of $H_\infty$ uncovered. Since each $k$-space 
  covers either $q^{k-1}$, $q^{k-2}$, \rev{or $0$} points of $H_i\backslash H_\infty$, the number of uncovered points in 
  $H_i\backslash H_\infty$ is divisible by $q^{k_{l-1}-2}$ if $l\ge 2$ and by $q^{k_l-1}$ if $l=1$, where $1\le i\le q$ is arbitrary. Let $a$ be the number of $k_l$-spaces 
  in $\cU$ that are completely contained in $H_i$, so that the number of uncovered points in $H_i$ equals 
  $$
    x\coloneq a\cdot q^{k_l-1}+(\#\cT-a)\cdot q^{k_l-2}.
  $$       
  If $l\ge 2$ we have $x\equiv 0 \pmod {q^{k_{l-1}-2}}$ and $\#\cT\equiv 0 \pmod {q^{k_{l-1}-k_l}}$, so that $(q-1)a\equiv 0\pmod {q^{k_{l-1}-k_l}}$ and $a\equiv 0\pmod {q^{k_{l-1}-k_l}}$. 
  If $l=1$ we have $x\equiv 0 \pmod {q^{k_{l}-1}}$ and $\#\cT\equiv 0 \pmod q$, so that $(q-1)a\equiv 0\pmod q$ and $a\equiv 0\pmod q$.   
\end{proof}

$\Delta$-divisible (multi-)sets $\cS$ of $k$-spaces in $\PG(n-1,q)$ have been studied in \cite{kurz2018heden}. If we replace each $k$-space by its $\tfrac{q^k-1}{q-1}$ points we obtain 
a $\Delta q^{k-1}$-divisible multiset of $\#\cS\cdot\tfrac{q^k-1}{q-1}$ points in $\PG(n-1,q)$. The possible cardinalities, given the divisibility constant and the field size, have 
been completely characterized in \cite[Theorem 1]{kiermaier2020lengths}. Here we will use only a few results on the possible structure of the tail (or more precisely of $\cT'$) which allow 
more direct proofs.
%% \begin{lemma}
%%   Let $\cU$ be an {\avsp} of $\PG(n-1,q)$ with tail $\cT$. If $\#\cT=q$, then either $\cU$ is reducible or we have $\cU=\cT$ and \rev{$\dim(U)=n-1$ for all $U\in\cU$}.
%% \end{lemma}
\begin{lemma}
  Let $\cU$ be an {\avsp} of $\PG(n-1,\rev{2})$ with tail $\cT$. If $\#\cT=\rev{2}$, then either $\cU$ is reducible or we have $\cU=\cT$ and \rev{$\dim(U)=n-1$ for all $U\in\cU$}.
\end{lemma} 
\begin{proof}
  Denote the dimension of the elements of $\cT$ by $k$. Lemma~\ref{lemma_divisible_sets_of_k_spaces} yields that $\cT'\coloneq\left\{T\rev{\cap} H_\infty\,:\, T\in\cT\right\}$ 
  is a $\rev{2}$-divisible multiset of $(k-1)$-spaces. So, each hyperplane of $H_\infty$ contains either all $\rev{2}$ or zero elements from $\cT'$, so that $\cT'$ is a $\rev{q}$-fold $(k-1)$-space. 
  With this, the stated results follows from Lemma~\ref{lemma_irreducible_no_multiset} \rev{applied to $\cT$}.  
\end{proof}

\begin{corollary}
  \label{corollary_no_tail_of_size_q}
  Let $\cU$ be an irreducible {\avsp} of $\PG(n-1,q)$ of type $k_1^{m_1}\dots k_l^{m_l}$, where $k_1>\dots>k_l$ and $k_i>0$ for all $1\le i\le l$. 
  If $m_l=q$, then we have $l=1$ and $k_1=n-1$.
\end{corollary}

\subsection{The structure of the tail for small parameters}
\label{subsec_tail}
If $\#\cT$ is small \rev{and $q=2$}, then we can also characterize the tail \rev{using Lemma~\ref{lemma_divisible_sets_of_k_spaces}}. To this end, let $\cS$ denote a set of $k$-spaces in $\PG(n-1,q)$. The corresponding \emph{spectrum} 
$\left(a_i\right)_{i\in\N_0}$ is given by the numbers $a_i$ of hyperplanes that contain exactly $i$ elements from $\cS$, so that
\begin{equation}  
  \sum_{i=0}^{\#\cS} a_i =\frac{q^n-1}{q-1}.\label{se1}
\end{equation}
The condition that $\cS$ is spanning, i.e.\ $\left\langle S\,:S \in \cS\right\rangle=\PG(n-1,q)$, is equivalent to $a_{\#\cS}=0$. Double-counting the $k$-spaces gives
\begin{equation}
  \sum_{i=0}^{\#\cS} i a_i =\#\cS\cdot \frac{q^{n-k}-1}{q-1}.\label{se2}
\end{equation}
  
\begin{lemma}
  \label{lemma_tail_4}
  Let $\cS$ be a $2$-divisible set of four $k$-spaces in $\PG(n-1,2)$. Then there exists a $(k-1)$-space $B$, a plane $E$, and a line $L\le E$ with 
  $\dim(\langle E,B\rangle)=k+2$, such that $\cS=\left\{ \langle P,B\rangle\,:\, P\in E\backslash L\right\}$.
\end{lemma}
\begin{proof}
  Assume that $P$ is a point that is contained in at least one but not all elements from $\cS$. Let $x$ denote the number of elements of $\cS$ that contain $P$. 
  Since all hyperplanes contain an even number of elements from $\cS$ we have $x\neq 3$. Assume $x=2$ for a moment and let $S,S'\in\cS$ be the two elements not containing $P$. 
  There are $2^{n-k-1}$ hyperplanes that contain $S$ but do not contain $P$, so that all of those hyperplanes contain $S$ and $S'$. The intersection of these hyperplanes has 
  dimension at most $k$ and contains $S$ as well as $S'$, so that $S=S'$, which is a contradiction. Thus, each point $P$ in $\PG(n-1,2)$ is contained in $0$, $1$ or $4$ elements of 
  $\cS$. 

  By $\left(a_i\right)_{i\in \N_0}$ we denote the spectrum of $\cS$. W.l.o.g.\ we assume that $\cS$ is spanning, i.e., we have $a_4=0$. From the equations~(\ref{se1}) and (\ref{se2}) 
  we conclude
  $$
    a_0 \,=\, 2^{n}-2^{n-k+1}+1\quad\text{and}\quad
    a_2 \,=\, 2^{n-k+1}-2.
  $$
  If there is no point $P$ that is contained in all four elements of $\cS$, then the elements of $\cS$ are pairwise disjoint and double-counting pairs yields
  \begin{equation}
    \binom{2}{2} a_2 =\binom{4}{2}\cdot \left(2^{n-2k}-1\right),\label{se3}
  \end{equation}
  so that
  $$
    2^{n-k+1}-2 =6\cdot \left(2^{n-2k}-1\right)\quad\Leftrightarrow\quad  2^{n-k}-3\cdot 2^{n-2k}+2=0,
  $$
  which has the unique solution $n=3$, $k=1$.
  
  So, by recursively quotienting out points $P$ that are contained in all elements of $\cS$ we conclude the existence of a $(k-1)$-space $B$ that is contained in all four elements of $\cS$. 
  Quotienting out $B$ yields a spanning $2$-divisible set of points in $\PG(2,2)$ with $a_0=1$ and $a_2=6$. Choosing $E$ as the ambient space and $L$ as the empty hyperplane yields the 
  stated characterization since in $\PG(2,2)$ there are exactly four points outside a hyperplane.
\end{proof}   
If $k=1$, i.e., the $k$-spaces are points, the equations (\ref{se1})-(\ref{se2}) \rev{and the generalization $\sum_{i=0}^{\#\cS} \binom{i}{2} a_i =\binom{\#\cS}{2}\cdot \frac{q^{n-2}-1}{q-1}$ of (\ref{se3})} are also known as {\lq\lq}standard equations{\rq\rq} or the first three MacWilliams equations for the
corresponding \rev{(projective)} linear code.

We remark that Lemma~\ref{lemma_tail_4} is based on the fact that each $2$-divisible set of $4$ points is an affine plane. For $q>2$ there there further possibilities 
for $q$-divisible sets of $q^2$ points over $\F_q$, see \cite{de2019cylinder,kurz2021generalization} on the so-called cylinder conjecture. \rev{We apply Lemma~\ref{lemma_tail_4} e.g.\ in Proposition~\ref{prop_tiavsp_pg_6_2} and Lemma~\ref{lemma_exclusion_6_1_5_4_4_4}.}  

\section{Classification of {\avsp}s in \texorpdfstring{$\mathbf{\PG(n-1,q)}$}{PG(n-1,q)} for small parameters}
\label{sec_classification_avsp}
By definition, there is no {\avsp} in $\PG(1-1,q)$. In $\PG(2-1,q)$ there is a unique {\avsp}. It has type $1^q$ and is irreducible and tight. 

\begin{lemma}
  \label{lemma_intersection_hyperplanes_as_avsp_elements}
  Let $\cU$ be an {\avsp} of $\PG(n-1,q)$, where $n\ge 3$. If there exist pairwise different hyperplanes $U_1,\dots,U_l\in \cU$, then there exists 
  an $(n-2)$-space $K\le H_\infty$ such that $K\le U_i$ for all $1\le i\le l$. 
\end{lemma}
\begin{proof}
  The statement is trivial for $l\le 1$, so that we assume $l\ge 2$. Due to the dimensions we have $\dim(U_i\cap U_j)=n-2$ for all $1\le i<j\le l$. Since 
  the sets of points $U_i\backslash H_\infty$ and $U_j\backslash H_\infty$ are disjoint we have $U_i\cap U_j\le H_\infty$ and $U_i\cap U_j=U_i\cap H_\infty=U_j\cap H_\infty$. 
  So, we set $K=U_1\cap H_\infty$.  
\end{proof}

\begin{proposition}
  Let $\cU$ be an irreducible {\avsp} of $\PG(n-1,q)$, where $n\ge 3$. If $\cU$ is of type $(n-1)^{m_{n-1}}\dots 2^{m_2}1^{m_1}$, then we have 
  $m_{n-1}\le q-2$ or $m_{n-1}=q$. In the latter case $\cU$ is not tight.
\end{proposition}
\begin{proof}
  We assume $m_{n-1}= q-1\ge 1$ and let $K\le H_\infty$ as in Lemma~\ref{lemma_intersection_hyperplanes_as_avsp_elements}. With this, let $H\neq H_\infty$ be the unique hyperplane 
  with $K\le H$ that is not contained as an element in $\cU$ and $\cU'$ arise from $\cU$ by removing the $q-1$ $(n-1)$-dimensional elements. Thus, $\cU'$ is an {\avsp} of $H$, i.e., $\cU$ is reducible.
  
  If $m_{n-1}=q$, then the $(n-2)$-space $K\le H_\infty$ (as in Lemma~\ref{lemma_intersection_hyperplanes_as_avsp_elements}) is contained in all elements of $\cU$, i.e., 
  $\cU$ is not tight.    
\end{proof}

\begin{corollary}
  \label{cor_m_n_minus_1_zero}
  Let $\cU$ be an irreducible tight {\avsp} of $\PG(n-1,2)$ of type $(n-1)^{m_{n-1}}\dots 1^{m_1}$, where $n\ge 3$. Then we have $m_{n-1}=0$.
\end{corollary}

Let $\cU=\left\{U_1,\dots,U_r\right\}$ be an {\avsp} of $\PG(n-1,q)$, $I\subseteq\{1,\dots,r\}$, and $V$ be a proper subspace with $V\not\le H_\infty$. 
If $\# I\ge 2$ and $\left\{U_i\,:\, i\in I\right\}$ is an {\avsp} of $V$, then we say that the spaces $U_i$ with $i\in I$ can be \emph{joined} to
$V$. Note that this is exactly the situation when $\cU$ is reducible. In $\PG(n-1,2)$ any two points outside of $H_\infty$ can be joined to a line, so that:   
\begin{lemma}
  \label{lemma_m_1_zero}
  Let $\cU$ be an irreducible tight {\avsp} of $\PG(n-1,2)$ of type $(n-1)^{m_{n-1}}\dots 1^{m_1}$, where $n\ge 3$. Then,we have $m_{1}=0$.
\end{lemma}

\begin{theorem}
  Let $\cU$ be an {\avsp} of $\PG(3-1,q)$ with type $2^{m_2} 1^{m_1}$. Then, we have $0\le m_2\le q$, $m_1=q\cdot\left(q-m_2\right)$, all 
  lines in $\cU$ contain a common point $P\le H_\infty$, and the $1$-dimensional elements can be grouped into pairwise disjoint sets of size $q$ that can be 
  joint to a line each.
\end{theorem}
\begin{proof}
  \rev{The parameterization of $m_2,m_1$ follows from the 
  packing condition~(\ref{eq_packing_condition}). If $m_2>0$, the}
  existence of $P$ follows from Lemma~\ref{lemma_intersection_hyperplanes_as_avsp_elements}. If $m_2=0$ then choose an arbitrary point $P\le H_\infty$. By $L_1,\dots, L_q$ we denote the $q$ lines 
  containing $P$ that are not equal to $H_\infty$. For each line $L_i$ that is not an element of $\cU$ there exist $q$ points in $\cU$ that can be joined to 
  $L_i$. (Note that $L_i\cap L_j=P$ for all $1\le i<j\le q$.) 
\end{proof}
We remark that all possibilities for $0\le m_2\le q$ can indeed by attained. In general there exist several non-isomorphic examples.

\begin{corollary}$\,$\\[-5mm]
  \begin{enumerate}
  \item[(1)] Let $\cU$ be an irreducible {\avsp} of $\PG(3-1,q)$. Then $\cU$ is of type $2^q$ and non-tight.\\[-4mm]
%% \end{corollary}
%%
%% \begin{corollary}
  \item[(2)] Let $\cU$ be an {\avsp} of $\PG(3-1,q)$ with type $2^{m_2} 1^{m_1}$. Then, $\cU$ is tight iff $m_1\ge 1$. In that case $\cU$ is reducible.\\[-4.5mm]
%% \end{corollary}
%% 
%% \begin{corollary}
  \item[(3)] No irreducible tight {\avsp} of $\PG(3-1,q)$ exists.
  \end{enumerate}
\end{corollary}  

Let $\cU$ be an {\avsp} of $\PG(n-1,q)$, where $n\ge 3$, and $K\le H_\infty$ be an arbitrary $(n-2)$-space. We say that 
$\cU^{(1)},\dots,\cU^{(q)}$ is a \emph{$K$-decomposition} of $\cU$ if the $q$ hyperplanes containing $K$ and not being equal to $H_\infty$ can be 
labeled as $H_1,\dots,H_q$ such that \rev{$\cU=\cup_{i=1}^q \cU^{(i)}$ and } 
\begin{equation}
  \cU^{(i)}=\left\{ U\cap H_i \,:\, U\in \cU, U\cap H_i\not\le H_\infty \right\}
\end{equation}  
for all $1\le i\le q$. Note that $\cU^{(i)}$ is an {\avsp} of $H_i$ for each $1\le i\le q$ (including the case $\cU^{(i)}=\left\{ H_i\right\}$). Moreover, any labeling of the $q$ 
hyperplanes $H_i$ induces a $K$-decomposition. Observe that for a fixed $(n-2)$-space $K\le H_\infty$ each pair of $K$-decompositions arises just by 
relabeling, so that we also speak of \emph{the} $K$-decomposition of $\cU$ since the actual labeling will not matter in our context.

\begin{proposition}
  \label{prop_at_least_one_hyperplane}
  Let $\cU$ be an {\avsp} of $\PG(n-1,q)$, where $n\ge 3$, with type $(n-1)^{m_{n-1}}\dots 2^{m_2}1^{m_1}$. If $1\le m_{n-1}\le q$, then 
  there exists an $(n-2)$-space $K\le H_\infty$ such that the $K$-decomposition $\cU^{(1)},\dots,\cU^{(q)}$ partitions $\cU$, i.e.,
  \[
    \bigcup_{1\le i\le q} \cU^{(i)}=\cU.
  \]
  Moreover, if $m_{n-1}\le q-1$, then $\cU$ is reducible. (More precisely, for each index $1\le i\le q$ with $\#\cU^{(i)}>1$ the elements in $\cU^{(i)}$ 
  can be joined to $H_i$.)
\end{proposition}
\begin{proof}
  Choose some arbitrary $U\in \cU$ with $\dim(U)=n-1$, set $K\coloneq U\cap H_\infty$, and let $\cU^{(1)},\dots,\cU^{(q)}$ be the $K$-decomposition of $\cU$ 
  and $H_1,\dots,H_q$ be the corresponding hyperplanes. Due to Lemma~\ref{lemma_intersection_hyperplanes_as_avsp_elements} each $U\in \cU$ with $\dim(U)=n-1$ 
  results in the same $(n-2)$-space $K$ and the same $K$-decomposition $\cU^{(1)},\dots,\cU^{(q)}$ (up to relabeling). Especially we have that for each 
  $U'\in\cU$ with $\dim(U')=n-1$ there exists an index $1\le i\le q$ with $\cU^{(i)}=\left\{ U'\right\}$. W.l.o.g.\ we assume $\#\cU^{(1)}=1$. 
  
  Now consider an element $U\in\cU$ with $\dim(U)<n-1$. From Lemma~\ref{lemma_characterization_splitted} we conclude $U\cap H_\infty\le K$ since 
  otherwise $\#\cU^{(1)}>1$ (more precisely, $U$ would split into $q$ $(\dim(U)-1)$-spaces where one of these would be contained in $\cU^{(1)}$ that
  also contains an entire hyperplane, which contradicts the packing condition~(\ref{eq_packing_condition})), which would contradict our assumption. 
  Thus, for each $U\in \cU$ there exists exactly one index $1\le i\le q$ with $U\in \cU^{(i)}$ and for each index $1\le j\le q$ either $U\in H_j$ or 
  $U\cap H_j\le K\le H_\infty$.      
\end{proof}                
 
\begin{corollary} 
  Let $\cU$ be an {\avsp} of $\PG(n-1,q)$ with type $(n-1)^{m_{n-1}}\dots 2^{m_2}1^{m_1}$, where $n\ge 3$. 
  If $\cU$ is irreducible, then we have $m_{n-1}\rev{\in\{0,q\}}$.
\end{corollary}  
    
 Affine vector space partitions of $\PG(4-1,q)$ that contain at least one hyperplane as an element can be characterized easily.   
\begin{proposition} 
  Let $\cU$ be an {\avsp} of $\PG(4-1,q)$ of type $3^{m_3}2^{m_2}1^{m_1}$ with $m_3\ge 1$. Then, we have 
  $1\le m_3\le q$, $0\le m_2 \rev{\le} q\cdot\left(q-m_3\right)$, and $m_1=q^3-q^2m_3-qm_2$. Moreover, there exists an $(n-2)$-space $K\le H_\infty$ 
  such that the $K$-decomposition $\cU^{(1)},\dots,\cU^{(q)}$ partitions $\cU$, so that $\cU$ especially is reducible if $m_{n-1}\le q-1$.
\end{proposition}   
\begin{proof}
  The equation $m_1=q^3-q^2m_3-qm_2$ directly follows from the packing condition~(\ref{eq_packing_condition}) and the ranges 
  $0\le m_2\le q\cdot\left(q-m_3\right)$, $0\le m_3\le q$ follow from the non-negativity 
  of $m_1,m_2,m_3$. Note that we have $m_3\ge 1$ by assumption. The remaining part follows from Proposition~\ref{prop_at_least_one_hyperplane}.
\end{proof}    

\section{Classification of tight irreducible {\avsp}s in \texorpdfstring{$\mathbf{\PG(n-1,2)}$}{PG(n-1,2)} for small dimensions \texorpdfstring{$\mathbf{n}$}{n}}
\label{sec_classification_ti_avsp}
The cases $n\le 3$ have already been treated in Section~\ref{sec_classification_avsp}, so that we assume $n\ge 4$ in the following. Our aim is 
to classify all possible types $(n-1)^{m_{n-1}}\dots 1^{m_1}$ such that a tight irreducible {\avsp} $\cU$ exists in $\PG(n-1,2)$. 
We have $m_{n-1}=0$ and $m_1=0$ due to Corollary~\ref{cor_m_n_minus_1_zero} and Lemma~\ref{lemma_m_1_zero}. From \rev{Corollary}~\ref{corollary_no_tail_of_size_q} 
we conclude $m_l\neq 2$ for the smallest index $1\le l\le n-1$ with $m_l>0$. The possible vectors $\left(m_{n-2},\dots,m_2\right) \in\N_0^{n-3}$ are 
quite restricted by the packing condition~(\ref{eq_packing_condition}). For $n=4$ the only remaining possibility is type $2^4$. From Lemma~\ref{lemma_irreducible_no_multiset} 
we conclude that the four lines are pairwise disjoint, i.e., they form a partial line spread of cardinality $4$. It is well known that each 
partial line spread of cardinality $q^2$ in $\PG(3,q)$ can be extended to a line spread, which has size $q^2+1$.\footnote{One argumentation is based on the 
fact that each $q^k$-divisible (multi-) set of $\tfrac{q^{k+1}-1}{q-1}$ points forms a $(k+1)$-space for each positive integer $k$, see e.g.\ \cite{honold2018partial}.}   
For $q=2$ there is only the Desarguesian line spread and since it has a transitive automorphism group, there is only one equivalence class. The numbers of line spreads in $\PG(3,q)$ 
are $1$, $2$, $3$, $21$, $1347$ for $q=2,3,4,5,7$. 

In the three subsequent subsections we will consider tight irreducible {\avsp}s in $\PG(n-1,2)$ for $n\in\{5,6,7\}$. The possible types are completely determined in all cases, where 
realizations are computed using an integer linear programming (ILP) formulation, see Section~\ref{sec_ilp} in the appendix for the details. If the sizes of the {\avsp}s are not too large 
we were able to also compute all equivalence classes of {\avsp}s using a slight modification of an algorithm from \cite{linton2004finding}, 
see also \cite[Algorithm 4.5]{kaski2006classification}. %% Section 4.2.2 
A \texttt{GAP} implementation , based on the \texttt{GAP} package {\lq\lq}FinInG{\rq\rq} \cite{fining} for computations in finite incidence geometry, can be obtained from 
the authors upon request. In the theoretical parts we will also use classification for $2$-divisible sets points that can e.g.\ be found in 
\cite{ubt_eref40887} or \cite{kurz2021divisible}. For the convenience of the reader we will also give a few selected proofs in Section~\ref{sec_details} in the appendix.

\subsection{Tight irreducible {\avsp}s in \texorpdfstring{$\mathbf{\PG(4,2)}$}{PG(4,2)}}
\label{subsec_pg_4_2}
We may use Lemma~\ref{lemma_divisible_sets_of_k_spaces} and Lemma~\ref{lemma_tail_4} to conclude that each {\avsp} of $\PG(n-1,2)$ of type $(n-2)^4$ is 
non-tight if $n\ge 5$. However, we can further tighten the statement to:
\begin{proposition}
  \label{lemma_three_times_co_dim_two}
  Let $\cU=\left\{U_1,\dots,U_r\right\}$ be an {\avsp} of $\PG(n-1,2)$, where $n\ge 4$, $r\ge 4$, and $\dim(U_i)=n-2$ for all $1\le i\le 3$. 
  Then, the elements $\left\{U_4,\dots,U_r\right\}$ can be joined to an $(n-2)$-space $B$ (including the case $r=4$ and $U_4=B$) and there exists an $(n-4)$-space 
  $C$ that is contained in all elements of $\left\{U_1,U_2,U_3,B\right\}$.
\end{proposition}
\begin{proof}
  First we assume that two elements of $\left\{U_1,U_2,U_3\right\}$ can be joined to an $(n-1)$-space $H$. Without loss of generality, we assume that $U_1$ and $U_2$ can be joined to $H$. 
  Let $K\coloneq H\cap H_\infty$, so that $\dim(K)=n-2$. By $H'$ we denote the unique hyperplane containing $K$ that is not equal to $H$ or $H_\infty$. Observe that 
  $\left\{U_3,\dots,U_r\right\}$ is an {\avsp} of $H'$ and $K$ is {\lq\lq}the hyperplane at infinity{\rq\rq} of $H'$. Next we set $K'\coloneq K\cap U_3$, 
  so that $\dim(K')=n-3$. Let $B$ denote the unique $(n-2)$-space in $H'$ that contains $K'$ and is not equal to $U_3$ or $K$. With this, $\left\{U_4,\dots,U_r\right\}$ 
  is an {\avsp} of $B$ (including the case $r=4$, $U_4=B$). Note that the $(n-3)$-space $K'$ is contained in all elements of $\left\{H,U_3,B\right\}$. Since 
  $\left\{U_1,U_2\right\}$ forms an {\avsp} of $H$ and $\dim(U_1)=\dim(\rev{U_2})=n-2$, there exists an $(n-4)$-space $C\rev{=K'\cap L}$ that is contained in all elements of $\left\{U_1,U_2,U_3,B\right\}$, \rev{where $L=U_1\cap U_2\subset K$ is an $(n-3)$-space}.
  
  Otherwise, we assume that no two elements of $\left\{U_1,U_2,U_3\right\}$ can be joined to an $(n-1)$-space, so that $\dim(U_i\cap U_j)=n-4$ for all $1\le i<j\le 3$. 
  We set $E_i\coloneq U_i\cap H_\infty$, so that $\dim(E_i)=n-3$, for all $1\le i\le 3$ and $\dim(E_i\cap E_j)=n-4$ for all $1\le i<j\le 3$. Let 
  $K\coloneq \left\langle E_1,E_2,E_3\right\rangle\le H_\infty$, so that $n-2\le\dim(K)\le n-1$. If $\dim(K)=n-2$, then consider the $K$-decomposition $\cU^{(1)},\cU^{(2)}$ of $\cU$ 
  and let $H_1,H_2$ be the corresponding hyperplanes. Since $E_1,E_2,E_3\le K$, we have that either $U_i\le H_1$ or $U_i\le H_2$ for all indices $1\le i\le 3$. By the pigeonhole 
  principle two of the three $(n-2)$-spaces in $\cU$ have to be contained in the same hyperplane,  \rev{which contradicts 
  $\dim(U_i\cap U_j)=n-4$.}
  
  %% so that they can be joined, which is a contradiction. 
  Thus, we have 
  $\dim(K)=n-1$, i.e., $K=H_\infty$. Since $\dim\left(\left\langle E_1, E_2\right\rangle\right)=n-2$, $\dim(E_3)=n-3$, and $\dim(K)=n-1$, we have 
  $$
    \dim(\left\langle E_1, E_2\right\rangle\cap E_3)=n-4.
  $$   
  Since $\dim(E_1\cap E_3)=\dim(E_2\cap E_3)=n-4$, we have $\dim(C)=n-4$ for $C\coloneq E_1\cap E_2\cap E_3$. Pick three linearly independent vectors $v_1,v_2,v_3$ such that 
  $E_1=\langle C,v_1\rangle$, $E_2=\langle C,v_2\rangle$, $E_3=\langle C,v_3\rangle$, and $H_\infty=\langle C,v_1,v_2,v_3\rangle$. Let $P_1,P_2$ be two different arbitrary 
  points outside of $H_\infty$ that or not covered by $U_1$, $U_2$, or $U_3$. For pairwise different $i,j,h\in\{1,2,3\}$ consider the $(n-2)$-space 
  $K_{i,j,j}\coloneq \left\langle C,v_i,v_j+v_h\right\rangle$ and let $H_{i,j,j}$ be the hyperplane that contains $K_{i,j,h}$ and $U_i$. Since all points 
  in $H_{i,j,h}\backslash H_\infty$ are covered by $U_1,U_2,U_3$ the points $P_1,P_2$ have to be contained in the other hyperplane containing $K_{i,j,h}$ not equal 
  to $H_{i,j,h}$ and $H_\infty$, so that $P_1-P_2\in \left\langle C,v_i,v_j+v_h\right\rangle$. Since
  $$
    \left\langle C,v_1,v_2+v_3\right\rangle \cap \left\langle C,v_2,v_1+v_3\right\rangle \cap \left\langle C,v_3,v_1+v_2\right\rangle =
    \left\langle C,v_1+v_2+v_3\right\rangle,  
  $$ 
  the $2^{n-3}$ points outside of $H_\infty$ that are not covered by $U_1$, $U_2$, or $U_3$ have to form an affine subspace $B\ge C$. If $\#\cU=r=4$, then 
  $B=U_4$. If $\#\cU\rev{\ge}5$, then the elements in $\left\{U_4,\dots,U_r\right\}$ form an {\avsp} of $B$. 
\end{proof}

\begin{corollary}
\label{cor_n_minus_2_ub}
  Let $\cU$ be an irreducible tight {\avsp} of $\PG(n-1,2)$ of type $(n-2)^{m_{n-2}}\dots 2^{m_2}$, where $n\ge 5$. Then, we have 
  $m_{n-2}\le 2$.
\end{corollary}

\noindent
Together with the conditions $m_{n-1}=m_1=0$ and the packing condition~(\ref{eq_packing_condition}) we obtain: 
\begin{corollary}
  Let $\cU$ be an irreducible tight {\avsp} of $\PG(5-1,2)$. Then the type of $\cU$ is given by $3^2 2^4$, $3^1 2^6$, or $2^8$.
\end{corollary}
All types can indeed be realized and corresponding numbers of equivalence classes are given by $3$, $4$, and $2$, respectively. I.e., for $n=4$ we have 
$9$ non-isomorphic examples in total. Below are representatives:

\smallskip

\noindent
{\scriptsize
E1, $2^8$:  $\< 10000,01000\>$, $\< 10100,00010\>$, $\< 11100,00001\>$, $\< 10010,01100\>$, $\< 11010,00101\>$, $\< 10001,01010\>$, $\< 10011,00110\>$, $\< 10111,\allowbreak 01110\>$.\\
\rev{E2}, $2^8$: $\< 10000,01000\>$, $\< 10100,00010\>$, $\< 11100,00001\>$, $\< 10010,01101\>$, $\< 10001,01011\>$, $\< 11001,00111\>$, $\< 10101,01110\>$, $\< 10011,\allowbreak 00100\>$.\\
E3, \rev{$3^1 2^6$}: $\< 10000,01000\>$, $\< 10100,00010\>$, $\< 11100,00001\>$, $\< 10010,01100\>$, $\< 10001,01010\>$, $\< 10111,01101\>$, $\< 10011,01010,00110\>$.\\
E4, \rev{$3^1 2^6$}: $\< 10000,01000\>$, $\< 10100,00010\>$, $\< 11100,00001\>$, $\< 10010,01100\>$, $\< 11001,00011\>$, $\< 10011,00100\>$, $\< 10001,01010,00100\>$.\\ 
E5, \rev{$3^1 2^6$}: $\< 10000,01000\>$, $\< 10100,00010\>$, $\< 11100,00001\>$, $\< 10010,01001\>$, $\< 10001,01111\>$, $\< 10111,01101\>$, $\< 10011,01010,00110\>$.\\
\rev{E6}, \rev{$3^1 2^6$}: $\< 10000,01000\>$, $\< 10100,00010\>$, $\< 11100,00001\>$, $\< 10010,01100\>$, $\< 10001,00100\>$, $\< 11001,00011\>$, $\< 10011,01000,00100\>$.\\
E7, \rev{$3^2 2^4$}: $\< 10000,01000\>$, $\< 10100,00010\>$, $\< 11100,00001\>$, $\< 10010,01011\>$, $\< 11010,00100,00001\>$, $\< 10001,00100,00010\>$.\\
E8, \rev{$3^2 2^4$}: $\< 10000,01000\>$, $\< 10100,00010\>$, $\< 11100,00001\>$, $\< 10010,01011\>$, $\< 10001,01010,00100\>$, $\< 10011,01001,00100\>$.\\
E9, \rev{$3^2 2^4$}: $\< 10000,01000\>$, $\< 10100,00010\>$, $\< 11100,00001\>$, $\< 10101,01011\>$, $\< 10010,01000,00001\>$, $\< 10001,01000,00110\>$.
}

\smallskip 
\noindent
We remark that the hypothetical type $3^3 2^2$ is also excluded by Corollary~\ref{corollary_no_tail_of_size_q}.  

\medskip 
 
Similarly as we have constructed $\cT'$ from the tail $\cT$ in Lemma~\ref{lemma_divisible_sets_of_k_spaces}, we can consider the set 
$\cU'\coloneq \left\{ U\cap H_\infty \,:\, U\in\cU\right\}$ for an {\avsp} $\cU$ of $\PG(n-1,q)$. If $\cU$ is an irreducible tight 
{\avsp} of $\PG(n-1,q)$ of type $2^{m_2} 3^{m_3}$, where $m_2=q^{\rev{n-2}}-qm_3$, then $\cU'$ is a configuration of $m_2$ points 
and $m_3$ lines in $H_\infty\cong \PG(n-2,q)$. The points are pairwise disjoint, so that Lemma~\ref{lemma_divisible_sets_of_k_spaces} yields that they 
form a $q$-divisible set. Any two lines can meet in at most a point. If $n=5$, then any two lines 
indeed intersect in a point. So, the maximum point multiplicity is at most $m_3+1$. 
\rev{We remark that the possibilities for $\cU'$ can be classified completely theoretically, i.e., without the use of computer programs. Due to space limitations we refer the interested reader to the corresponding arXiv preprint and only state two necessary criteria for $\cU'$.}
%% In the following we will theoretically classify the possibilities for $\cU'$ 
%% for tight irreducible {\avsp}s of $\PG(4,2)$ of type $2^{m_2} 3^{m_3}$ up to symmetry. Corollary~\ref{cor_n_minus_2_ub} gives $m_3\in\{0,1,2\}$. 
%% First we will deduce two general necessary criteria for $\cU'$.    
\begin{lemma} 
  \label{lemma_spanning}
  Let $\cU$ be an irreducible {\avsp} of $\PG(n-1,q)$ not of type $(n-1)^q$ and $\cU'\coloneq \left\{U\cap H_\infty\,:\,U\in\cU\right\}$. Then $\cU'$ is spanning, i.e., 
  $\cU'$ spans $H_\infty$.
\end{lemma}
\begin{proof}
  Assume that $K$ is a hyperplane of $H_\infty$ that contains all elements of $\cU'$. From Lemma~\ref{lemma_characterization_full} we can conclude that the 
  $K$-decomposition $\cU^{(1)},\dots,\cU^{(q)}$, with corresponding hyperplanes $H_1,\dots,H_q$, is a partition of $\cU$, i.e., the elements of $\cU^{(i)}$ can 
  be joined to $H_i$ for all $1\le i\le q$. Since we have assumed that $\cU$ is not of type $(n-1)^q$ we obtain a contradiction.
\end{proof}
\begin{lemma} 
  \label{lemma_partition_packing}
  Let $\cU$ be an {\avsp} of $\PG(n-1,q)$ of type $(n-1)^{m_{n-1}}\dots 2^{m_2}$ and $\cU'\coloneq \left\{U\cap H_\infty\,:\,U\in\cU\right\}$. For each  
  hyperplane $K$ of $H_\infty$ let $a_i^K$ denote the number of $i$-dimensional elements of $\cU'$ that are contained in $K$ and $b_i^K=m_{i+1}-a_{i}^K$ the number of 
  $i$-dimensional elements of $\cU'$ that are \rev{not} contained in $K$, where $1\le i\le n-2$. Then there exist $c_{i,j}^K\in\N_0$ for all $1\le j\le q$, $1\le i\le n-2$  such that
  \begin{eqnarray}
    \sum_{j=1}^q c_{i,j}^K =a_{i}^K\quad \forall 1\le i\le n-2,\label{eq_part_pack_1}\\
    \sum_{i=1}^{n-2} \left(c_{i,j}^K\cdot q^{i} \,+\, b_i^K \cdot q^{i-1}=q^{n-2}\right)     \quad \forall 1\le j\le q. \label{eq_part_pack_2}
  \end{eqnarray}  
\end{lemma}
\begin{proof}
  For an arbitrary but \rev{fixed} hyperplane $K$ of $H_\infty$ let $\cU^{(1)},\dots,\cU^{(q)}$ be the $K$-decomposition of $\cU$ with corresponding hyperplanes $H_1,\dots,H_q$.
  From Lemma~\ref{lemma_characterization_full} we conclude that for each element $U\in \cU$ with $U\cap H_\infty\le K$ there exists an index $1\le j\le q$ such that $U\le H_j$. 
  The \rev{integers} $c_{i,j}^K$ just count how many $(i+1)$-dimensional elements \rev{of $\cU$} are contained in $H_j$ \rev{(which depends on $K$)}. Since the hyperplanes $H_1,\dots,H_q$ are pairwise disjoint, we obtain 
  Equation~(\ref{eq_part_pack_1}). From Lemma~\ref{lemma_characterization_splitted} we conclude that for each element $U\in\cU$ such that $U\cap H_\infty\not\le K$ we have 
  $\#\left(\rev{(} U\cap H_j\rev{)} \backslash H_\infty\right)=q^{\dim(U)-2}$, so that the packing condition \rev{for $H_j$} yields Equation~(\ref{eq_part_pack_2}).   
\end{proof}

We call the process of moving from 
$\cU'$ to $\cU$ the \emph{extension problem}. An integer linear programming formulation is given in Section~\ref{sec_ilp} in the appendix. 
Note that the extension problem comprises additional symmetry given by the pointwise stabilizer of $H_\infty$ of order $q^{n-1}$.  
%% we can write $\PG(n-1,q)=\left\langle H_\infty,P\right\rangle$ for an arbitrary point 
%% outside of $H_\infty$. When factoring out this symmetry several extensions collapse to the same equivalence class of {\avsp}s.

Given a set $\cU'$ satisfying all of the necessary conditions mentioned so far it is neither clear that an extension to a corresponding {\avsp} $\cU$ always exists nor
that it is, in the case of existence, unique up to symmetry. Indeed, we will give counter examples later on. However, for the nine classified configurations $\cU'$ in $\PG(3,2)$ it turns out 
that there always is an up to symmetry unique extension.   

\subsection{Tight irreducible {\avsp}s in \texorpdfstring{$\mathbf{\PG(5,2)}$}{PG(5,2)}}
\label{subsec_pg_5_2}
%%Let us continue the classification of the possible types of tight irreducible {\s of $\F_2$. The next case is $\PG(5,2)$.
\begin{lemma}
  \label{lemma_no_type_4_2_3_4}
  For $n\ge 6$ no tight irreducible {\avsp} of type $(n-2)^2 (n-3)^4$ in  $\PG(n-1,2)$ exists.
\end{lemma}
\begin{proof}
  Assume that such an {\avsp} $\cU$ exists and consider the intersections of the elements with the hyperplane $H_\infty$ at infinity, i.e., 
  $\cU'\coloneq \left\{U\cap H_\infty\,:\, U\in\cU\right\}$. By $E_1,E_2$ we denote the two $(n-3)$-spaces and by $L_1,\dots,L_4$ the four $(n-4)$-spaces. The intersection of 
  $E_1$ and $E_2$ is an $(n-4)$-space $L'$ and $\dim(E_i \rev{\cap} L_j)\ge n-5$ for all $i=1,2$ and $j=1,\dots,4$. From Lemma~\ref{lemma_divisible_sets_of_k_spaces} we conclude 
  that $\cT'=\left\{L_1,\dots,L_4\right\}$ is a $2$-divisible set of four $(n-4)$-spaces, so that Lemma~\ref{lemma_tail_4} implies the existence of a plane $E\le H_\infty$,
  a line $L\le E$, and an $(n-5)$-space $B\le H_\infty$ with $B\cap E=\emptyset$ and
  $$
    \left\{L_1,\dots,L_4\right\}=\left\{\langle Q,B\rangle\,:\, Q\in E\backslash L\right\}.
  $$
  Since $\cU$ is tight we have $B\cap L' =\emptyset$. However, $\dim(E_i \rev{\cap} L_j)\ge n-5$ implies $\dim(L',L_j)\ge n-6$ for all $1\le j\le 4$. So, we clearly have 
  $n\le 7$.

  \rev{For $n=7$ we have $\dim(E_i\cap B)\ge \dim(E_i \cap L_j)-1\ge n-6=1$, where $1\le i\le 2$ and $1\le j\le 4$, so that $\dim(L'\cap B)=0$ implies $E_i\le K:=\langle B,L'\rangle$ and $\dim(E_i\cap B)=1$. With this, 
  $\dim(E_i \cap L_j)\ge n-5=2$ yields the existence of a 
  point $Q_{i,j}\notin B$ with $Q_{i,j}\le E_i\cap L_j$, so 
  that $L_j=\langle B,Q_{i,j}\rangle$, i.e., $L_j\le K$. However, this implies that $\cU'$ is not spanning, which is a contradiction with Lemma~\ref{lemma_spanning}}.
  
  %% For $n=7$ we conclude $E\le E_1$, $E\le E_2$, $\dim(B)=2$, $\dim(E_1)=\dim(E_2)=4$, $\dim(E)=3$, and $\dim(L')=3$, so that $L'=E$ and $\dim(E_i\cap B)\ge 1$ for $i=1,2$. 
  %% Thus, we have $\langle \cU'\rangle \le \langle E,B\rangle$, i.e., $\cU'$ is not spanning, which is a contradiction \rev{with Lemma~\ref{lemma_spanning}}. 

  For $n=6$ we have $\dim(B)=1$, i.e., $B$ is a point. Since $\cU$ is tight we have $B\not\le L'$. W.l.o.g.\ we assume $B\not\le E_1$. \rev{Let $S:=\langle E_1,B\rangle$, so that $\dim(S)=4$. Since $E_1$ intersects each of the lines 
  $L_j$ in at least a point not equal to $B$, we have $L_j\le S$ for all $1\le j\le 4$. Let $1\le h\le 4$ be a suitable index with $L_h\neq L'$ and let $Q\le E_2\cap L_h$ be a point, so that $Q\not\le L'$. Thus $E_2=\langle L',Q\rangle\le S$, so that $\langle \cU'\rangle\le S$, 
  i.e., $\cU'$ is not spanning, which is a contradiction with Lemma~\ref{lemma_spanning}.} 
  %% For $n=6$ we have $\dim(B)=1$, i.e., $B$ is a point. Since $\cU$ is tight we have $B\not\le L'$. W.l.o.g.\ we assume $B\not\le E_1$. Since $E_1$ intersects each of the lines 
  %% $L_j$ in at least a point, we have $E_1=E$. Since $\cU$ is irreducible $E_2$ is not contained in the solid $S\coloneq \langle E,B\rangle$, \rev{see Lemma~\ref{lemma_spanning}}. Since $E_1$ intersects each of the lines 
  %% $L_j$ in at least a point, we have that the line $L'\le E$ intersects each of the lines $L_j$ in at least a point. Since $B\not\le L'$ this is impossible.
\end{proof}
  
\begin{proposition}
  \label{prop_tiavsp_pg_5_2}
  Let $\cU$ be a tight irreducible {\avsp} of $\PG(5,2)$, then $\cU$ has one of the following types:
  \begin{itemize}
    \item $4^2 3^i 2^{8-2i}$ for $i\in\{0,1,2\}$;
    \item $4^1 3^i 2^{12-2i}$ for $i\in\{0,\dots,6\}\backslash\{5\}$; and
    \item $3^i 2^{16-2i}$ for $i\in\{0,\dots,8\}\backslash\{7\}$.
  \end{itemize}
  All types are realizable.
\end{proposition}
\begin{proof}
  Let the type of $\cU$ be $5^{m_5}\dots 1^{m_1}$. From Corollary~\ref{cor_m_n_minus_1_zero} and Lemma~\ref{lemma_m_1_zero} we conclude $m_5=0$ and $m_1=0$, so that the 
  packing condition (\ref{eq_packing_condition}) gives $4m_4+2m_3+m_2=16$. Corollary~\ref{cor_n_minus_2_ub} gives $m_4\le 2$ and Lemma~\ref{lemma_no_type_4_2_3_4} 
  excludes $\left(m_4,m_3,m_2\right)=(2,4,0)$. Moreover, Corollary~\ref{corollary_no_tail_of_size_q} implies $m_2\neq 2$. All remaining possibilities 
  $\left(m_4,m_3,m_2\right)\in\mathbb{N}_0^3$ are listed in the statement and for each type we found a realization using ILP computations. 
  %% an ILP formulation, see Section~\ref{sec_ilp} in the appendix. 
\end{proof}  
\begin{corollary}
  If $\cU$ is a tight irreducible {\avsp} of $\PG(5,2)$ of minimum possible size, then $\#\cU=7$ and $\cU$ has type $4^1 3^6$.
\end{corollary}

For small sizes we have enumerated the \rev{isomorphy} %%isomorphism 
types of tight irreducible {\avsp}s in $\PG(5,2)$, see Table~\ref{table_isomorphism_types_pg_5_2}. The last row concerns the parts 
$\cU'$ at the hyperplane $H_\infty$ at infinity w.r.t.\ the {\avsp}s $\cU$ counted up to isomorphy in the second row. So, for e.g.\ types 
$4^1 3^4 2^4$ and $4^2 2^8$ there exist configurations $\cU'$ that allow more than one extension up to symmetry. 
%% An explicit example:
%% [[[[1,0,1,0,0,0,],[0,1,0,0,0,0,],], 
%% [[1,0,1,1,0,0,],[0,0,0,0,1,0,],], 
%% [[1,0,0,1,1,0,],[0,1,1,0,1,0,],], 
%% [[1,1,0,1,1,0,],[0,0,1,0,0,0,],], 
%% [[1,0,0,0,0,0,],[0,1,0,0,0,0,],[0,0,0,1,0,0,],], 
%% [[1,0,0,0,1,0,],[0,1,0,0,0,0,],[0,0,1,0,0,0,],], 
%% [[1,0,0,0,0,1,],[0,1,0,0,0,0,],[0,0,1,0,0,0,],[0,0,0,1,0,0,],[0,0,0,0,1,0,],], 
%% ], 
%% [[[1,0,0,0,0,1,],[0,0,0,1,1,0,],], 
%% [[1,1,0,0,0,1,],[0,0,1,1,1,0,],], 
%% [[1,0,0,1,0,1,],[0,0,1,0,1,0,],], 
%% [[1,1,0,1,0,1,],[0,0,0,0,1,0,],], 
%% [[1,0,1,0,0,1,],[0,1,0,0,0,0,],[0,0,0,1,0,0,],], 
%% [[1,0,0,0,1,1,],[0,1,0,0,0,0,],[0,0,1,0,0,0,],], 
%% [[1,0,0,0,0,0,],[0,1,0,0,0,0,],[0,0,1,0,0,0,],[0,0,0,1,0,0,],[0,0,0,0,1,0,],], 
%% ], 
%% ];

\begin{table}[htp]
  \begin{center}
    \begin{tabular}{l|l|ll|ll|ll}
      \hline 
      type & $4^1 3^6$ & $4^2 3^2 2^4$ & $3^8$ & $4^2 3^1 2^6$ & $4^1 3^4 2^4$ & $4^2 2^8$ & $4^1 3^3 2^6$ \\ 
      \#   & 6         & 38            & 32    & 55            & 827           & 83        & 8096 \\
      in $H_\infty$ & 6& 38            & 32    & 55            & 811           & 50        & 6686 \\
      \hline
    \end{tabular}
    %%\caption{Number of isomorphism types of tight irreducible {\avsp}s in $\PG(5,2)$ per type.}
    \caption{Number of isomorphism types of tight irreducible {\avsp}s in $\PG(5,2)$.}
    \label{table_isomorphism_types_pg_5_2}
  \end{center}
\end{table}

For the minimum possible size of a tight irreducible {\avsp} in $\PG(5,2)$ we can write down all implications of 
the stated necessary conditions for the part $\cU'$ at infinity. So, for type $4^1 3^6$ configuration $\cU'$ consists of one plane $E$ 
and six lines $\mathcal{L}=\left\{L_1, \dots,L_6\right\}$ satisfying the following conditions:
\begin{enumerate}
  \item[(1)] the configuration is spanning, i.e., $\left\langle E,L_1,\dots,L_6\right\rangle=\PG(4,2)$;\\[-6mm]
  \item[(2)] the configuration is tight, i.e., there does not exist a point $P$ that is contained in $E$ and all lines in $\mathcal{L}$;\\[-6mm]
  \item[(3)] the lines in $\mathcal{L}$ form a $2$-divisible set of lines, i.e., each hyperplane contains an even number of lines;\\[-6mm]
  \item[(4)] each line $L_i$ intersects $E$ in at least a point;\\[-6mm]
  \item[(5)] hyperplanes that contain $E$ also contain at least two lines.
\end{enumerate}  
Up to symmetry ten such configurations exist:
  \begin{enumerate} \scriptsize
    \item[E1:] 
$\langle 10000,01000,00100\rangle$,  
$\langle 10000, 01000\rangle$, 
$\langle 10000, 00100\rangle$, 
$\langle 01000, 00010\rangle$, 
$\langle 01000, 00110\rangle$, 
$\langle 10100, 00001\rangle$, 
$\langle 10100, 01101\rangle$\\[-5mm]
    \item[E2:]
$\langle 10000, 01000, 00100\rangle$, 
$\langle 10000, 01000\rangle$, 
$\langle 10000, 00010\rangle$, 
$\langle 10000, 00110\rangle$, 
$\langle 01100, 00010\rangle$, 
$\langle 01100, 00001\rangle$, 
$\langle 10011, 01100\rangle$\\[-5mm]    
    \item[E3:]
$\langle 10000, 01000, 00100\rangle$, 
$\langle 10000, 01000\rangle$,
$\langle 10000, 00010\rangle$, 
$\langle 10000, 00001\rangle$, 
$\langle 01000, 00011\rangle$, 
$\langle 10100, 01011\rangle$, 
$\langle 01011, 00111\rangle$\\[-5mm]    
    \item[E4:]
$\langle 10000, 01000, 00100\rangle$, 
$\langle 10000, 01000\rangle$,
$\langle 10000, 00010\rangle$, 
$\langle 01000, 00110\rangle$, 
$\langle 00100, 00010\rangle$, 
$\langle 11100, 00001\rangle$, 
$\langle 10111, 01011\rangle$\\[-5mm]    
    \item[E5:]
$\langle 10000, 01000, 00100\rangle$, 
$\langle 10000, 00010\rangle$, 
$\langle 10000, 01010\rangle$, 
$\langle 00100, 00001\rangle$, 
$\langle 01100, 00011\rangle$, 
$\langle 11001, 00100\rangle$, 
$\langle 10111, 01100\rangle$\\[-5mm] 
    \item[E6:]
$\langle 10000, 01000, 00100\rangle$, 
$\langle 10000, 00010\rangle$, 
$\langle 10000, 00001\rangle$, 
$\langle 10000, 01011\rangle$, 
$\langle 01000, 00010\rangle$, 
$\langle 01000, 00001\rangle$, 
$\langle 10011, 01000\rangle$\\[-5mm] 
    \item[E7:]
$\langle 10000, 01000, 00100\rangle$, 
$\langle 10000, 00010\rangle$, 
$\langle 10000, 00001\rangle$, 
$\langle 10000, 01011\rangle$, 
$\langle 01000, 00010\rangle$, 
$\langle 01000, 00101\rangle$, 
$\langle 10111, 01000\rangle$\\[-5mm] 
    \item[E8:]
$\langle 10000, 01000, 00100\rangle$, 
$\langle 10000, 00010\rangle$, 
$\langle 10000, 00001\rangle$, 
$\langle 01000, 00010\rangle$, 
$\langle 01000, 00001\rangle$, 
$\langle 10100, 01111\rangle$, 
$\langle 10111, 01100\rangle$\\[-5mm]    
    \item[E9:]
$\langle 10000, 01000, 00100\rangle$, 
$\langle 10000, 00010\rangle$, 
$\langle 10000, 00001\rangle$, 
$\langle 01000, 00110\rangle$, 
$\langle 01000, 00101\rangle$, 
$\langle 10100, 01111\rangle$, 
$\langle 10111, 01100\rangle$\\[-5mm]    
    \item[E10:]
$\langle 10000, 01000, 00100\rangle$, 
$\langle 10000, 00010\rangle$, 
$\langle 01000, 00010\rangle$, 
$\langle 00100, 00001\rangle$, 
$\langle 10100, 01011\rangle$, 
$\langle 11001, 00101\rangle$, 
$\langle 10011, 01100\rangle$     
\end{enumerate} 
It turns out that E2, E4, E7, and E9 are not extendable to an {\avsp} while the other six cases are. Moreover, the extension is unique up to symmetry in these cases.

\subsection{Tight irreducible {\avsp}s in \texorpdfstring{$\mathbf{\PG(6,2)}$}{PG(6,2)}}
\label{subsec_pg_6_2}
%% Next we treat the possible types of tight irreducible affine vector space partitions in $\PG(6,2)$. 
\begin{lemma}
  \label{lemma_exclusion_avsp_6_2}
  In $\PG(6,2)$ no tight irreducible {\avsp} of type $5^2 4^2 3^4$ or $5^1 4^6$ exists.
  %% In $\PG(6,2)$ no tight irreducible {\avsp} of the following types exist: $5^2 4^4$, $5^2 4^2 3^4$, $5^1 4^6$. 
\end{lemma}
\begin{proof}
  All two %%three 
  possibilities are excluded using ILP computations, see Section~\ref{sec_ilp}. They are also excluded using \texttt{GAP} computations. 
\end{proof}
% 5^2 * excluded | two 5-spaces prescribed
% MIP - Integer optimal solution:  Objective =  1.0000000000e+01
% Solution time = 53236.81 sec.  Iterations = 717325762  Nodes = 15686136
% Deterministic time = 9191461.48 ticks  (172.65 ticks/sec)

\begin{proposition}
  \label{prop_tiavsp_pg_6_2}
  Let $\cU$ be a tight irreducible {\avsp} of $\PG(6,2)$, then $\cU$ has one of the following types:
  \begin{itemize}
    \item $5^2 4^i 3^j 2^{16-2j-4i}$ for $i\in\{0,1,2\}$ and $0\le j\le 8-2i$, where $j+2i \neq 7$ and $(i,j)\neq (2,4)$;
    \item $5^1 4^i 3^j 2^{24-2j-4i}$ for $0\le i\le 4$ and $0\le j\le 12-2i$, where $j+2i \neq 11$;
    \item $4^i 3^j 2^{32-2j-4i}$ for $0\le i\le 8$ and $0\le j\le 16-2i$, where $j+2i \neq 15$ and $i\neq 7$.
  \end{itemize}
  All types are realizable.
\end{proposition}
\begin{proof}
  Let the type of $\cU$ be $6^{m_6}\dots 1^{m_1}$. From Corollary~\ref{cor_m_n_minus_1_zero} and Lemma~\ref{lemma_m_1_zero} we conclude $m_6=0$ and $m_1=0$, so that the packing 
  condition (\ref{eq_packing_condition}) gives $8m_5+4m_4+2m_3+m_2=32$. Corollary~\ref{cor_n_minus_2_ub} yields $m_5\le 2$ and Lemma~\ref{lemma_no_type_4_2_3_4} 
  excludes $\left(m_5,m_4,m_3,m_2\right)=(2,4,0,0)$. Moreover, Corollary~\ref{corollary_no_tail_of_size_q} implies $m_l\neq 2$ for the smallest index with $m_l>0$, \rev{which excludes the types $5^2 4^j 3^i 2^2$ with $j+2i=7$, $5^2 4^3 2^2$, $5^1 4^j 3^i 2^2$ with $j+2i=11$, $5^1 4^5 3^2$, $4^j 3^i 2^2$ with $j+2i=15$, and $4^7 3^2$}. The two 
  hypothetical types $5^2 4^2 3^4$ and  $5^1 4^6$ are excluded in Lemma~\ref{lemma_exclusion_avsp_6_2}. \rev{For the three hypothetical types $5^2 4^3 2^4$, $5^1 4^5 2^4$, and $4^7 2^4$ we apply Lemma~\ref{lemma_divisible_sets_of_k_spaces} to conclude that the set of $2$-spaces is $4$-divisible. However, Lemma~\ref{lemma_tail_4} characterizes the $2$-divisible sets of cardinality $4$ and we can easily check that it is not $4$-divisible.} All remaining possibilities $\left(m_5,m_4,m_3,m_2\right)\in\N_0^4$ 
  are listed in the statement and for each type we found a realization using an ILP formulation, see Section~\ref{sec_ilp}. %% in the appendix. 
\end{proof}
  
\begin{corollary}
  If $\cU$ is a tight irreducible {\avsp} of $\PG(6,2)$ of minimum possible size, then $\#\cU=8$ and $\cU$ has type $4^8$.
\end{corollary}  
  
Here we describe all four isomorphism types of homogeneous irreducible tight 
{\avsp}s $\cU$ of $\PG(6,2)$ of type $4^8$, \rev{where we call an {\avsp} $\cU$ homogeneous if all of its elements have the same dimension}. Geometrically each $\cU$ is given by 
eight solids $S_1,\dots,S_8$ in $\PG(6,2)$ intersecting a hyperplane $H_\infty$ in a plane (plus some extra conditions). Here we directly consider 
the part $\cU'$ at infinity, i.e.\ the eight planes $\pi_1,\dots,\pi_8\in H_\infty\cong \PG(5,2)$ given by $\pi_i=S_i\cap H_\infty$. The conditions 
for the pairwise intersections are 
\begin{equation}
  \label{cond_intersection_planes}
  1\le \dim(\pi_i\cap \pi_j)\le 2\quad\forall  1\le i<j\le 8.
\end{equation}
Since the planes form a spanning $2$-divisible set we have 
\begin{equation}
  \label{cond_incidences_planes_hyperplanes}
  \#\left\{1\le i\le 8\,:\, \pi_i\not\le H\right\}\in\{2,4,6,8\}  
\end{equation}
for every hyperplane $H$ of $H_\infty\cong\PG(5,2)$.  

Let $e_i$ denote the $i$th unit vector, i.e., the vector with a $1$ at the $i$-th position and zeros everywhere else. If the pairwise intersection of 
the planes $\pi_i$ is a line in all cases then they span a solid, which contradicts the condition that not all eight planes can be contained in a hyperplane. 
W.l.o.g.\ we assume $\pi_1=\langle e_1,e_2,e_3\rangle$ and $\pi_2=\langle e_3,e_4,e_5\rangle$, i.e., the intersection point between $\pi_1$ and $\pi_2$ is 
$\langle e_3\rangle$. Since the intersection of all eight planes is empty we assume w.l.o.g.\ that $\pi_3$ does not contain $\pi_1\cap \pi_2=e_3$. Up to 
symmetry we have the following three cases for $\pi_3$:
\begin{itemize}
  \item[(a)] $\dim(\pi_1\cap\pi_3)=2$, $\dim(\pi_2\cap\pi_3)=1$: $\pi_3=\langle e_1,e_2,e_4\rangle$;
  \item[(b)] $\dim(\pi_1\cap\pi_3)=\dim(\pi_2\cap\pi_3)=1$, $\dim(\langle \pi_1,\pi_2,\pi_3\rangle)=5$: $\pi_3=\langle e_1,e_4,e_2+e_5\rangle$; and 
  \item[(c)] $\dim(\pi_1\cap\pi_3)=\dim(\pi_2\cap\pi_3)=1$, $\dim(\langle \pi_1,\pi_2,\pi_3\rangle)=6$: $\pi_3=\langle e_1,e_4,e_6\rangle$.
\end{itemize}  
Starting from the three possibilities for $\pi_1,\pi_2,\pi_3$ we build up a graph whose vertices consist of the planes that have intersection dimension $1$ or 
$2$ with $\pi_i$ for $1\le i\le 3$, cf.\ Condition~(\ref{cond_intersection_planes}). Two vertices $\pi$ and $\pi'$ are connected by an 
edge if $1\le \dim(\pi\cap\pi')\le 2$, cf.\ Condition~(\ref{cond_intersection_planes}). For these graphs we determine all cliques of size five and check 
Condition~(\ref{cond_incidences_planes_hyperplanes}) afterwards:
 \begin{itemize}
  \item[(a)] $3{,}014{,}435{,}152$ cliques $\rightarrow$ $432$~cases;\\[-6mm] %%in 310~seconds;\\[-6mm]
  \item[(b)] $2{,}198{,}293{,}872$ cliques $\rightarrow$ $0$~cases;\\[-6mm] %% in 215~seconds;\\[-6mm]
  \item[(c)] $1{,}218{,}975{,}648$ cliques $\rightarrow$ $320$~cases. %% in 140~seconds.
\end{itemize}
The overall computation took just a few minutes. 
Note that the constructed $752$~cases are just candidates for the extension problem to eight solids. Up to symmetry they decompose into just four non-isomorphic 
examples. It turns out that they can be distinguished by the maximum number $\gamma_0$ of incidences of a point and the eight planes, \rev{where $2\le \gamma_0\le 5$}.
In Table~\ref{table_constructions_4_8_in_pg_6_2} we summarize incidence counts, \rev{i.e., for $X\in\{\text{point, line, solid,hyperplane}\}$ the stated vector ${a_1}^{b_1}\dots{a_r}^{b_r}$ says that $b_i$ of the $X$s have exactly $a_i$ incidences with the eight planes, given the isomorphy type characterized by $\gamma_0$. The last row states how often the {\lq\lq}third plane{\rq\rq} is of type (a), (b), or (c) after fixing a pair of planes $\pi_1,\pi_2$.} 

\begin{table}[htp]
  \begin{center}
    \begin{tabular}{|l|llll|}
      \hline 
      $\gamma_0$ & 2 & 3 & 4 & 5 \\
      \hline
      point \rev{incidences:}       & $2^{28}$  & $1^{21}2^7 3^7$ & $1^{16}2^{12} 4^4$ & $1^{20} 2^6 3^2 4^2 5^2$ \\
      line incidences:       & $1^{56}$  & $1^{56}$        & $1^{48}2^4$        & $1^{46} 2^2 3^2$         \\
      solid incidences:      & $1^{56}$  & $1^{56}$        & $1^{48}2^4$        & $1^{40} 2^8$             \\
      hyperplane incidences: & $2^{28}$  & $2^{28}$        & $2^{24}4^2$        & $2^{23} 4^1 6^1$         \\
      triples \rev{of planes $\pi_i$}:          & $c^{168}$ & $c^{168}$       & $a^{48}c^{96}$     & $a^{72} c^{48}$          \\ 
      \hline     
    \end{tabular}
    \caption{Irreducible tight {\avsp}s of $\PG(6,2)$ of type $4^8$.}
    \label{table_constructions_4_8_in_pg_6_2}
  \end{center}
\end{table}

For $\gamma_0=2$ we consider an arbitrary plane $\pi$ contained in the hyperbolic quadric $\cQ=Q^+(5,2)$, which form a single orbit under its collineation group $\mathrm{PGO}^+(6, 2) = C_2 \times \mathrm{PGL}(3,2) = S_8$ of order $40{,}320$.  
From the $35$ points on $\cQ$ the points in $\pi$ have no incidences with the eight planes while all other $28$ points on $\cQ$ have exactly two incidences. This example is 
obtained in 16~cases. The symmetry group of the eight planes has order $1344$ and type $C_2^3 : \operatorname{PGL}(3,2)$.

For $\gamma_0=3$ choose a projective base of $\PG(5,2)$, i.e., put $f_i = e_i$ for $1 \leq i \leq 6$ and $f_7 = \sum_{i=1}^6 e_i$. Consider a Fano plane on the set $\{ 1, 2, 3, 4, 5, 6, 7 \}$:
\begin{align*}
  &\ell_1 = \{ 1, 2, 3 \}, && \ell_2 = \{ 1, 4, 5 \}, && \ell_3 = \{ 1, 6, 7\}, &&\ell_4 = \{ 2, 4, 6 \}, \\
  & \ell_5 = \{ 3, 4, 7 \}, && \ell_6 = \{ 2, 5, 7 \}, && \ell_7 = \{ 3, 5, 6 \}.
\end{align*}
Choose seven planes $\pi_i \coloneq \< f_j: j \in \ell_i \>$ for $1\le i\le 7$ and an eight plane.
$\pi_8=K \coloneq \< \sum_{j \in \ell_i} f_j: 1 \leq \rev{i} \leq 7 \>$. Note that $K$ itself is also a Fano plane 
(of course with a different embedding). The points with three incidences with the eight planes are the $f_i$ for $1\le i\le 7$ and the points with two incidences 
with the eight planes are the points of $K$. This example is obtained in 112~cases. The symmetry group of the eight planes has order $168$ and type $\operatorname{PGL}(3,2)$.

For $\gamma_0=4$ let $\left\{ Q_1, Q_2, Q_3,Q_4,R_1, R_2 \right\}$ be a basis of $H_\infty$. 
%% For instance, we can take $\left(Q_1, Q_2, Q_3, Q_4, R_1, R_2\right) = \left(e_1, e_2, e_3, e_4, e_5, e_6\right)$. 
With this, we construct the eight planes as
\begin{align*}
  & \< Q_{i+j},Q_{i+j+1}, R_i \> \text{ for } i \in \{ 1,2\}\rev{,}  j\in \{0,2\},\\
  & \< Q_{i+j},Q_{i+j+1}, R_i+A \> \text{ for } i \in \{ 1,2\}. j\in \{0,2\},
\end{align*}
where $A=Q_1+Q_2+Q_3+Q_4$ and $Q_5=Q_1$. The points with four incidences with the eight planes are $Q_1,\dots,Q_4$. The lines with two incidences with the eight planes are 
$\langle Q_i,Q_{i+1}\rangle$ for $1\le i\le 4$ (again setting $Q_5=Q_1$; so this is some kind of a cyclic construction). This example is obtained in 192~cases. The symmetry 
group of the eight planes has order $128$ and type $D_8^2 : C_2$.

For $\gamma_0=5$ let $\left\{ Q_1, Q_2, R_1, R_2, S_1, T_1  \right\}$ be a basis of $H_\infty$. 
%% For instance, we can take $\left(Q_1, Q_2, R_1, R_2, S_1, T_1\right) = \left(e_1, e_2, e_3, e_4, e_5, e_6\right)$. 
With this, we set $S_2\coloneq S_1+Q_1+Q_2$, $T_2\coloneq T_1+Q_1+Q_2+R_1+R_2$ and construct the eight planes as
\begin{align*}
  & \< Q_1,Q_2, R_i \> \text{ for } i \in \{ 1,2\},\\
  &\< R_1,R_2, S_i \> \text{ for } i \in \{ 1,2\}, \rev{\text{and}}\\
  & \< Q_i, R_i, T_j \> \text{ for } i,j \in \{ 1,2\},
\end{align*}
which also reflects the three orbits of the eight planes w.r.t.\ the action of their automorphism group. The points with five incidences with the eight planes are $R_1$ 
and $R_2$. The points with four incidences with the eight planes are $Q_1$ and $Q_2$. The points with three incidences with the eight planes are $R_1+Q_1$ and $R_2+Q_2$. 
The lines with three incidences with the eight planes are $\langle R_1,Q_1\rangle$ and $\langle R_2,Q_2\rangle$. The lines with two incidences with the eight planes are 
$\langle R_1,R_2\rangle$ and $\langle Q_1,Q_2\rangle$. This example is obtained in 432~cases. The symmetry group of the eight planes has order $1024$. 

\section{Constructions of tight irreducible {\avsp}s}
\label{sec_constructions}
In this section we collect a few general constructions for tight irreducible {\avsp}s using different combinatorial objects. We use spreads, the Klein quadric, and hitting formulas in sections 
\ref{subsec_spreads}, \ref{subsec_klein_quadric}, and \ref{subsec_hitting_formulas}, respectively.

\subsection{Constructions from projective spreads}
\label{subsec_spreads}

A $k$-spread in $\PG(n-1,q)$ is a disjoint set of $k$-spaces that partitions $\PG(n-1,q)$. It is well known that $k$-spreads exist iff $k$ divides $n$.
\begin{proposition}
  For each positive even integer $n$ there exists a tight irreducible {\avsp} $\cU$ of $\PG(n-1,q)$ of type $(n/2)^m$, where $m=q^{n/2}$.
\end{proposition} 
\begin{proof}
  Let $k=n/2$ and $\cP$ be a $k$-spread of $\PG(n-1,q)$, which has size $q^k+1$. Now choose an arbitrary element $K\in\cP$ and an arbitrary hyperplane $H$ containing $K$.
  With this we set $\cU=\cP\backslash\{K\}$ where we choose $H$ as the hyperplane at infinity. By construction $\cU$ is an {\avsp} of $\PG(n-1,q)$. Since 
  all elements are pairwise disjoint $\cU$ is tight and since any two elements span $\PG(n-1,q)$ $\cU$ is irreducible.   
\end{proof}

We have seen that in $\PG(5,2)$ there exist tight irreducible {\avsp}s of types $3^8$ and $2^{16}$. Starting from a $2$-spread of $\PG(5,q)$ we can clearly 
obtain a tight {\avsp} $\cU$ by removing all lines that are completely contained in an arbitrarily chosen hyperplane $H$. However, it may happen that $\cU$ is reducible. 
This is indeed the case if we start with the Desarguesian line spread. In $\PG(5,2)$ there exist $131{,}044$ non-isomorphic line spreads \cite{mateva2009line}. 

\begin{conjecture}
  For each integer $1<k<n$ that divides $n$ there exists a tight irreducible {\avsp} $\cU$ of $\PG(n-1,q)$ of type $k^m$, where $m=q^{n-k}$.
\end{conjecture}

If $n$ is odd no $\left\lfloor (n-1)/2\right\rfloor$-spread exists, but we can construct tight irreducible {\avsp}s from some special large partial spreads.
\begin{proposition}
  For each odd integer $n\ge 5$ there exists a tight irreducible {\avsp} $\cU$ of $\PG(n-1,q)$ of type $((n-1)/2)^m$, where $m=q^{(n+1)/2}$.
\end{proposition} 
\begin{proof}
  Let $k=(n-1)/2$ and $\cP$ be a vector space partition of $\PG(n-1,q)$ of type $\rev{(k+1)^1} k^m$, where $m=q^{k+1}$. Now choose an arbitrary hyperplane $H$ containing the unique
  $(k+1)$-dimensional element $K$ of $\cP$. With this we set $\cU=\cP\backslash\{K\}$ where we choose $H$ as the hyperplane at infinity. By construction $\cU$ is an {\avsp} of $\PG(n-1,q)$. Since all elements are pairwise disjoint $\cU$ is tight. Any two elements of $\cU$ span a hyperplane of $\PG(n-1,q)$. Since the elements 
  of $\cU'\coloneq \left\{ U\cap H_\infty\,:\,U\in\cU\right\}$ span $H_\infty$, not all elements of $\cU'$ can be contained in a hyperplane of $H_\infty$ and $\cU$ is irreducible. 
\end{proof}
Vector space partitions of the used type can be obtained from lifted MRD codes, see e.g.\ \cite{sheekey201913} for a survey on MRD codes. They also occur as extendible partial $k$-spreads, 
where $k=(n-1)/2$, of the second largest size $q^{k+1}$ and are the main building block in the construction of partial $k$-spreads of size $q^{k+1}+1$ as described by 
Beutelspacher \cite{beutelspacher1975partial}. For more details on the relations between these different geometrical objects we refer e.g.\ to \cite{honold2019classification}.   

For each $n\ge 5$ there also exist a vector space partition $\cP$ of $\PG(n-1,q)$ of type $(n-2)^1 2^m$, where $m=q^{n-2}$. Choosing a hyperplane that contains the unique $(n-2)$-space
as the hyperplane at infinity we can obtain a tight {\avsp} $\cU$ of $\PG(n-1,q)$ of type $2^{q^{n-2}}$. The remaining question is whether we can choose $\cP$ in 
such a way that $\cU$ becomes irreducible.

\subsection{Constructions from the Klein quadric}
\label{subsec_klein_quadric}

It seems very likely that the {\avsp} of $\PG(6,2)$ of type $4^8$ with maximum point multiplicity $2$, see Subsection~\ref{subsec_pg_5_2}, can be generalized to arbitrary field sizes.

\begin{theorem}\label{theorem_hyperbolic_quadric}
  There exists a tight irreducible {\avsp} of type $4^{q^3}$ in \rev{$\PG(6, q)$} %%$\AG(6, q)$ 
  for $q$ even.
\end{theorem}

\begin{proof}
We will use the following finite field model of $\AG(6,q)$.
Let $V=\mathbb{F}_{q^3}\times \mathbb{F}_{q^3}\times \mathbb{F}_q$ and
let $H_\infty$ be the hyperplane $X_3=0$. So we identify $\AG(6,q)$ with
the elements of $V$ of the form $(a,b,c)$, where $c\ne 0$.
Consider the following quadratic form on $H_\infty$:
\[
Q(x,y,0)\coloneq Tr_{q^3/q}(xy).
\]
Then $Q$ defines the points of a hyperbolic quadric $\cQ$.
Next, let $\pi$ be the plane $\{(0,y,0):y\in\mathbb{F}_{q^3}^*\}$.
Then $\pi$ is totally singular with respect to $Q$. 
Let $S_0:=\{(x,0,1):x\in \mathbb{F}_{q^3}\}$ and 
$S_1\coloneq\{(y,y^{q^2}+y^q+1,1):y\in\mathbb{F}_{q^3}\}$.
%$S_1\coloneq\{(y+1,y^{q^2}+y^q+1,1):y\in\mathbb{F}_{q^3}\}$.
Let $\alpha$ be a primitive element of $\mathbb{F}_{q^3}$, let $\sigma$ be the map
    \[    
        \sigma: (x, y, z) \mapsto (\alpha^{-1} x, \alpha y, z),
    \]
and let $G:=\langle \sigma\rangle$. We will show that
$\cS:=\{S_0\}\cup S_1^G$ is a tight irreducible {\avsp} of size $q^3$ in $\AG(6, q)$.

First note that $\sigma$ has order $q^3-1$. 
Let $(a,b,1)$ be a point $P$ of $\AG(6, q)$.
We show that $P$ lies in a unique element of $\cS$.
If $b=0$, then $P$ lies in $S_0$.
The condition that $P$ lies in $S_1^{\sigma^m}$ (where $1\le m \le q^3-1$) 
can be restated as
    \begin{align*}
   a = \alpha^{-m} y, &&  b = \alpha^m(y^{q^2}+y^q+1).
    \end{align*}
for some $y\in\mathbb{F}_{q^3}$. 
We have 
\[
ab=y^{q^2+1}+y^{q+1}+y
\]
and the polynomial $y^{q^2+1}+y^{q+1}+y$ is a permutation
on $\mathbb{F}_{q^3}$, by \cite[Theorem 4]{TuZengHu} \rev{(using our assumption $q$ even)}. Hence, 
$y$ and, thus, $m$ are determined by $a$ and $b$. Therefore, $\cS$ is an {\avsp}.

Note that
$\pi_1\coloneq\overline{S_1}\cap H_\infty=\{(y,y^{q^2}+y^q,0):y\in\mathbb{F}_{q^3}\}$
and $\pi_0\coloneq\overline{S_0}\cap H_\infty=\{(x,0,0):x\in\mathbb{F}_{q^3}\}$.
To compute the image of $\pi_1$ under $\sigma^m$, notice that
\begin{align*}
    (\alpha^{-m}y,\alpha^m(y^{q^2}+y^q),0)&=(\alpha^{-m}y, \alpha^{(q^2+1)m}(\alpha^{-m}y)^{q^2}+\alpha^{(q+1)m}(\alpha^{-m}y)^q,0)=(w, \zeta w^{q^2}+\zeta^q w^q,0)
\end{align*}
where $w=\alpha^{-m}y$ and $\zeta=\alpha^{(q^2+1)m}$.
Therefore, upon application of $G$,
\[
\cS_\infty \coloneq\{ \overline{S}\cap H_\infty : S \in \cS\} = \{\pi_0\}\cup \pi_1^G= \{ \pi_\zeta: \zeta \in \mathbb{F}_{q^3}\}
\]
where
$\pi_\zeta := \{ (y, \zeta y^{q^2}+\zeta^q y^q, 0 ) : y \in \mathbb{F}_{q^3}^* \}$.
Note that $|\cS_\infty| = q^3$ and that $\cS$ consists of totally singular 
planes of $\cQ$ disjoint from $\pi$. 
As these are all totally singular planes of $\cQ$ disjoint from $\pi$,
these cover the points of $\cQ$ uniformly and
their common intersection is empty and, thus, $\cS$ is tight.
As these pairwise meet 
in a point, any two elements of $\cS$ span $\PG(6, q)$.
This shows irreducibility.
\end{proof}

Let $\cP$ be the set of planes in the Klein quadric $\cQ=Q^+(5,q)$ that is disjoint to an arbitrary but fixed plane $\pi$ in $\cQ$. One can verify that $\cP$ is a spanning $q$-divisible set of $q^3$ planes in $\PG(5,q)$ such that the intersection of a pair of planes is a point, i.e., all known conditions for the part $\cU'$ at infinity of a tight irreducible {\avsp} of $\PG(6,q)$ of type 
$4^{q^3}$ are satisfied. The remaining question is whether a solution of the extension problem for $\cP$ exists.

\begin{conjecture}
  \label{conj_extension_klein_quadric}
  The extension problem for $\cP$ admits a solution for all prime powers $q$.
\end{conjecture}
Theorem \ref{theorem_hyperbolic_quadric} shows the conjecture for $q$ even. 
By computer we showed
Conjecture~\ref{conj_extension_klein_quadric} for $q=3,5$.

\subsection{Constructions using hitting formulas}
\label{subsec_hitting_formulas}
A \emph{hitting formula} is a DNF such that each truth assignment to the underlying variables satisfies precisely one term \cite{Iwama}. 
For example:
\[
 (x \land y \land z) \lor
 (\bar{x} \land \bar{y} \land \bar{z}) \lor
 (\bar{x} \land y) \lor
 (\bar{y} \land z) \lor
 (\bar{z} \land x).
\]
We say that a variable \emph{appears} in the DNF if one of the two corresponding literals appears in one of the terms. The variables mentioned in the above DNF are $x,y,z$. 
We can represent hitting formulas over $x_1,\ldots,x_n$ as collections of strings in $\{0,1,*\}^n$, where $0$ in the $i$'th position represents $\bar{x}_i$, $1$ in the $i$'th 
position represents $x_i$, and $*$ in the $i$'th position represents the absence of $x_i$ in the term. For example, the above hitting formula corresponds to the strings
%%\[
 $111,000,01*,*01,1{*}0$.
%%\]

This notion describes subcubes of affine points. Taking the projective closure we end up with the list
$$
  \langle 1111\rangle, \langle 1000\rangle, \langle 1010,0001\rangle, \langle 1001,0100\rangle , \langle 1100,0010\rangle
$$
of subspaces of $\PG(3,2)$ that form an {\avsp}, which obviously is not irreducible. However, we can join the first two elements to 
$\langle 1000,0111\rangle$ and obtain a tight irreducible {\avsp}. While every string corresponds to an affine subspace, not every affine 
subspace corresponds to a string. It turns out that any two strings having its stars at the same positions can be joined to an affine subspace. 
For brevity, we speak of \emph{compression}. Interestingly enough, several tight irreducible {\avsp}s of $\PG(n-1,2)$ of the minimum possible size
can be obtained by compression, see Section~\ref{sec_compression} in the appendix. More theoretical insights on the relations between hitting formulas 
and {\avsp}s \rev{can be found in \cite{filmus2022irreducible}--}  
%% will be treated in an upcoming article 
focusing on irreducible hitting formulas.

\section{The minimum possible size of tight irreducible {\avsp}s}  
\label{sec_minimum_size}
We have discussed the minimum possible size of a (tight) {\avsp} of $\PG(n-1,q)$ in Section~\ref{sec_preliminaries}. Before we consider the 
minimum possible size $\sigma_q(n)$ of a tight irreducible {\avsp} $\cU$ of $\PG(n-1,q)$ we remark that Lemma~\ref{lemma_m_1_zero} implies 
the upper bound $\#\cU\le 2^{n-2}$ for $q=2$. The constructions mentioned in Section~\ref{sec_constructions} suggest that this upper bound can be attained. 
In Section~\ref{sec_classification_avsp} and Section~\ref{sec_classification_ti_avsp} we have determined the exact values $\sigma_q(2)=q$, $\sigma_q(3)=\infty$, 
$\sigma_2(4)=4$, $\sigma_2(5)=6$, $\sigma_2(6)=7$, and $\sigma_2(7)=8$.

\begin{lemma}
  \label{lemma_exclusion_6_1_5_4_4_4}
  \rev{In $\PG(7,2)$ no tight irreducible {\avsp} of type $6^1 5^4 4^4$ exists.}
\end{lemma}
\begin{proof}
  \rev{Assume that $\cU$ is a tight irreducible {\avsp} of type $6^1 5^4 4^4$ in $\PG(7,2)$.  Consider $\cU':=\left\{U\cap H_\infty \mid U\in\cU\right\}$. From Lemma~\ref{lemma_divisible_sets_of_k_spaces}  and Lemma~\ref{lemma_tail_4} we conclude that the four planes in $\cU'$ share a common line $L$ and that there is a unique configuration up to symmetry. Since the $5$-space in $\cU'$ intersects each of the four planes in dimension at least $2$, it also intersects $L$ in dimension at least $1$. We enumerate the possible configurations of the $5$-space and the four planes in $\cU'$ up to symmetry. For each such configuration we build up a list of candidates for the four solids using the facts that the intersect the planes in dimension $1$ or $2$ and the $5$-space in dimension at least $3$ or $4$. Next we consider a $4$-subsets of those candidates whose dimensions of the pairwise intersections are contained in $\{2,3\}$. We end up with a list of candidates for $\cU'$.  Here we can eliminate those which a common point or are not spanning, cf.~Lemma~\ref{lemma_spanning}. For each hyperplane $H$ of $H_\infty$ let $s:=\left(s_3,s_4,s_5\right)$ be given by 
  $s_i:=\#\left\{ U\in\cU'\mid \dim(U)=i,U\le H\right\}$. 
  From Lemma~\ref{lemma_partition_packing} we can conclude that the following cases cannot occur:
  \begin{itemize}
    \item $s=(0,1,0)$;\\[-5mm]
    \item $s=(0,3,0)$;\\[-5mm]
    \item $s=(0,0,1)$;\\[-5mm]
    \item $s=(2,3,1)$;\\[-5mm]
    \item $s=(0,1,1)$.
  \end{itemize}  
  For the remaining cases we have checked computationally that the extension problem does not admit a solution.}
\end{proof}

\begin{lemma}
  In $\PG(6,2)$ every configuration $\cU'$ of type $5^2 4^2 3^4$, $5^2 4^1 3^6$, $5^1 4^6$, %%$5^1 4^4 3^4$, 
  or $4^8$ that  
  %% does not contain a configuration as in Lemma~\ref{lemma_special_configuration_1} and 
  satisfies the 
  conditions of Lemma~\ref{lemma_spanning}, Lemma~\ref{lemma_partition_packing}, and the dimension condition, cf.~Lemma~\ref{lemma_dimension_condition}, admits a point $P$ that is contained in all elements of $\cU'$. 
\end{lemma}
\begin{proof}
  All cases have been excluded by ILP computations, \rev{cf.~Section~\ref{sec_ilp} for general model formulations}.
\end{proof}
% verified using write_ILP_avsp_infty_part_[l].cpp for l=1,...,5

\begin{corollary}
  In $\PG(7,2)$ no tight irreducible {\avsp} of the following types exist: $6^2 5^2 4^4$, $6^2 5^1 4^6$, $6^1 5^6$, $6^1 5^4 4^4$, $5^8$. 
\end{corollary}

\begin{corollary}
  The minimum size $\sigma_2(8)$ of an irreducible tight {\avsp} of $\PG(7,2)$ is given by $10$.
\end{corollary}
\noindent
For attaining examples we refer to Section~\ref{sec_compression} in the appendix.

\medskip
\noindent
Our next aim is a recursive construction which implies an asymptotic upper bound of roughly $\tfrac{3n}{2}$ for $\sigma_2(n)$.    
\begin{theorem}
  \label{lemma_recursion_1}
  Let $\cU=\left\{U_1,\dots,U_r\right\}$ be an irreducible tight {\avsp} of $\PG(n-1,2)$ with $\dim(U_1)=n-2$ and $n\ge 3$. Then, 
  there exists an irreducible tight {\avsp} $\cU'$ of $\PG(n+2-1,2)$ of size $\#\cU+3=r+3$ that contains an element of dimension $n$.
\end{theorem}
\begin{proof}
  Let $V=\PG(n+2-1,2)$, $H_\infty$ be the hyperplane at infinity, and $K\le H_\infty$ be an arbitrary subspace with $\dim(K)=n$. With this, denote the two hyperplanes 
  containing $K$ and not being equal to $H_\infty$ by $H_1$ and $H_2$. Choose an arbitrary point $P\le K$ and a subspace $K'\le K$ such that $\dim(K')=n-1$ 
  and $\left\langle P,K'\right\rangle=K$. Now choose an irreducible tight {\avsp} $\cU=\left\{U_1,\dots,U_r\right\}$ of $H_1/P$ such that 
  $\dim(U_1)=n-2$. We set $A_i\coloneq \left\langle U_i,P\right\rangle$ for all $1\le i\le r$. Choose an $n$-space $B$ with $B\cap H_1=A_1$ and $B\not \le H_\infty$, so 
  that $C_1\coloneq B\cap H_2$ is an $(n-1)$-space in $H_2$ with $C_1\not\le H_\infty$ and $P\le C_1$. In $H_2$ choose three further $(n-1)$-spaces $C_2,C_3,C_4$ such that 
  $\dim(C_i\cap C_j)=\dim(C_1\cap C_2\cap C_3\cap C_4)=n-3$ for all $1\le i<j\le 4$, $C_1\cap C_2\cap C_3\cap C_4\le K'$, and that $\left\{ C_1,C_2,C_3,C_4\right\}$ 
  forms an {\avsp} of $H_2$. (This boils down to an {\avsp} of $\PG(4-1,2)$ of type $2^4$, which is a union of four disjoint lines.) Then, 
  $$
    \cU'\coloneq\left\{A_2,\dots,A_r,B,C_2,C_3,C_4\right\}
  $$
  is an irreducible tight {\avsp} of $V$ of size $\#\cU+3=r+3$. The size follows directly from the construction and $\dim(B)=n$. 
  Since $B\cap H_1\cap H_\infty=B\cap H_2\cap H_\infty$ we have $B\cap C_2\cap C_3\cap C_4 =C_1\cap C_2\cap C_3\cap C_4\le K'$ and 
  $B\cap A_2\cap\dots \cap A_r =A_1\cap\dots\cap A_r=P$, so that $\cU'$ is tight. Noting that $\cU''\coloneq\left\{ A_1,\dots,A_r\right\}$ is an {\avsp}
  of $H_1$, $\left\{C_1,\dots,C_4\right\}$ is an {\avsp} of $H_2$, and $\left\{A_1,C_1\right\}$ is an {\avsp} of $B$, we conclude that $\cU'$ is 
  indeed an {\avsp} of $V$. 
   
   It remains to show that $\cU'$ is irreducible. So, assume that there exists a proper subset $\tilde{\cU}\subsetneq \cU'$ that can be joined 
   to an $x$-space $X$. If $\tilde{\cU} \cap \left\{B,C_2,C_3,C_4\right\}=\emptyset$, then we have $\tilde{\cU}\subseteq \cU''$ contradicting the 
   fact that $\cU''$ is irreducible. So, especially we have $x\in\{n,n+1\}$. Noting that any two elements in $\left\{C_1,C_2,C_3,C_4\right\}$ span 
   $H_2$, we conclude $\#\left(\tilde{\cU} \cap \left\{B,C_2,C_3,C_4\right\}\right)=1$.      
   \begin{itemize}
     \item[(i)] If $x=n$, then let $2\le i\le 4$ be the unique index such $C_i\in \tilde{\cU}$. Clearly, $B\notin\tilde{\cU}$. Let $\tilde{C}$ be the the other $(n-1)$ 
                space in $X$ not contained in $H_\infty$ 
                and not equal to $C_i$ with $\tilde{C}\cap H_\infty=C_i\cap\tilde{C}$, so that the elements of $\tilde{\cU}\backslash \{C_i\}$ form a 
                vector space partition of $\tilde{C}$. However, since $P\not\le C_i$ and all elements in $\cU\backslash\left\{B,C_1,C_2,C_3\right\}$ contain $P$, 
                this is impossible.
     \item[(ii)] If $x=n+1$ and $\#\left(\tilde{\cU} \cap \left\{B,C_2,C_3,C_4\right\}\right)=1$, then we have $\dim(X')=n$ for $X'\coloneq X\cap H_1$. If 
                 $B\in \tilde{\cU}$, then $\tilde{\cU}\backslash\{B\} \cup \left\{A_1\right\}$ can be joined to $X'$ in $H_1$, which is a contradiction. 
                 If $C_i\in \tilde{\cU}$, then $\tilde{\cU}\backslash\{C_i\} $ can be joined to $X'$ in $H_1$, which is also a contradiction.            
   \end{itemize}
   Thus, $\cU'$ is irreducible.
\end{proof}
        
\begin{corollary}
  For each $n\ge 4$ an irreducible tight {\avsp} $\cU$ of $\PG(n-1,2)$ of size $\left\lfloor \tfrac{3n-3}{2}\right\rfloor$ exists.
\end{corollary}        
\begin{proof}
  For $n=4$ there exists such an example with type $2^4$ and for $n=5$ there exists such an example with type $3^2 2^4$. Then, iteratively apply the 
  construction from Theorem~\ref{lemma_recursion_1}.
\end{proof}        
\noindent
We remark that the constructive upper bound for $q_2(n)$ is tight for $n\in\{4,5,6,8\}$. 
  
\section{Conclusion}
\label{sec_conclusion}
We have introduced the geometrical object of affine vector space partitions. To make their study interesting we need the additional conditions of tightness 
and irreducibility, which are natural in the context of hitting formulas. A very challenging problem is the determination of the minimum possible size 
of an irreducible tight {\avsp} of $\PG(n-1,q)$. To this end we have obtained some preliminary results for arbitrary field sizes but small dimensions and 
for the binary case with medium sized dimensions. We also gave a parametric construction that matches the known exact values in many cases. That irreducible 
tight {\avsp}s are nice geometric objects can be e.g.\ seen at their sometimes large automorphism groups as well as the mentioned connection to the hyperbolic 
quadric $Q^+(5,q)$. While we have obtained a few insights, many questions remain open. So, we would like to close with a list of a few open problems:
\begin{enumerate}
  \item Consider tight irreducible {\avsp}s of $\PG(4,q)$ of type \rev{$3^{m_3} 2^{m_2}$}. What is the largest possible value for $m_3$?\\[-6mm] 
  \item Determine a solution of the extension problem for the set $\cP$ of $q^3$ planes in $\PG(5,q)$ obtained from the hyperbolic quadric $Q^+(5,q)$ for $q$ odd, cf.~Conjecture~\ref{conj_extension_klein_quadric}.\\[-6mm]
  \item Determine further constructions for tight irreducible {\avsp}s of $\PG(n-1,q)$ with large automorphism groups.\\[-6mm]
  \item Construct a tight irreducible {\avsp} of $\PG(n-1,q)$ of type $2^{q^{n-2}}$ for all $n\ge 5$.\\[-6mm]
  \item Is it possible that a tight irreducible {\avsp} of $\PG(n-1,q)$ contains $1$-dimensional elements if \rev{$n\ge 4$} and $q\ge 3$?\\[-6mm]
  \item Determine further exact values of the minimum size $\sigma_q(n)$ of a tight irreducible {\avsp} of $\PG(n-1,q)$.\\[-6mm]
  \item Determine $\lim_{n\to\infty} \sigma_q(n)/n$.\\[-6mm]
  \item Is $\sigma_q(n)$ strictly increasing in $n$?\\[-6mm]
\end{enumerate}  

\section*{Acknowledgements}
\rev{First of all we would like to thank the two anonymous referees for their careful reading and the many useful remarks that improved the presentation of the paper a lot.} 
We thank Esmeralda N{\u{a}}stase, Artur Riazanov, and Yuriy V.~Tarannikov for helpful discussions. %% and remarks on an earlier draft of this paper. 
Further thanks go to 
Tomáš Peitl and Stefan Szeider for sharing with us the results of the computer search reported in~\cite{PeitlSzeider}.
Ferdinand Ihringer and Sascha Kurz would like to thank the organizers of the Sixth Irsee Conference on Finite Geometries for their invitation. During that conference 
the idea of analyzing and introducing {\avsp}s slowly evolved, being triggered by some open problems for hitting formulas.

This project has received funding from the European Union's Horizon 2020 research and innovation programme under grant agreement No~802020-ERC-HARMONIC. Ferdinand Ihringer is supported by a postdoctoral fellowship of the Research Foundation -- Flanders (FWO).

\appendix
\section{Integer linear programming formulations}
\label{sec_ilp}

Let $\cU'$ be an arbitrary set of subspaces of $H_\infty$ in $\PG(n-1,q)$. For the question whether $\cU'$ can be extended to an 
{\avsp} $\cU$ of $\PG(n-1,q)$ we utilize binary variables $x_C$ for all subspaces $C$ of $\PG(n-1,q)$ such that $C\not\le H_\infty$ 
and $C\cap H_\infty \in \cU'$ with the meaning $x_C=1$ iff $C\in\cU$. We denote the set of all of these subspaces by $\cC$. For each 
point $P$ in $\PG(n-1,q)\backslash H_\infty$ the equation
\begin{equation}
  \sum_{\rev{C}\in\cC\,:\,P\le C} x_C =1   
\end{equation}
and for each $U\in\cU'$ the \rev{inequality}
\begin{equation}
  \label{ie_extension}
  \sum_{\rev{C}\in\cC\,:\,U\le C} x_C \rev{\ge} 1   
\end{equation}
has to be satisfied. (If we are only interested in irreducible {\avsp}s, then we can require {\lq\lq}$=${\rq\rq} in Inequality~(\ref{ie_extension}).) The $0/1$ solutions of this equation system are in one-to-one correspondence to extensions of $\cU'$ 
to {\avsp}s $\cU$ in $\PG(n-1,q)$.

\medskip

Searching a tight irreducible {\avsp} $\cU$ in $\PG(n-1,q)$ directly can be achieved by a similar model. Now let $\cC$ be the set of subspaces of $\PG(n-1,q)$ that 
are not incident with $H_\infty$. Again we use binary variables $x_C$ for all $C\in\cC$ with the meaning $x_C=1$ iff $C\in\cU$. Partitioning the affine points is modeled 
by 
\begin{equation}
  \sum_{\rev{C}\in\cC\,:\,P\le C} x_C =1   
\end{equation}
for all points $P$ not contained in $H_\infty$. The condition that $\cU$ is tight can be written as
\begin{equation}
  \sum_{C\in \cC\,:\,Q\le C} x_C +1 \le \sum_{C\in \cC} x_C 
\end{equation}
for all points $Q\le H_\infty$. In order to model the condition that $\cU$ is irreducible we say that a subspace $A$ \emph{escapes} a subspace $B$ if 
$A$ has both points that are contained and points that are not contained in $B$. So, for each $B\in \cC$ we require
\begin{equation}
  x_B\,+\, \sum_{C\in\cC\text{ such that $C$ escapes $B$}} x_C \ge 1,
\end{equation}
\rev{i.e., either $B\in\cU$ or there exists an element $C\in\cU$ certifying that no subset of $\cU$ can be joined to $B$.}

Of course we can fix the type of $\cU$ by additional equations. Using a target function we can minimize or maximize $\#\cU$ as well as the number 
of $i$-dimensional elements. We have to mention that this ILP formulation comprises a lot of symmetry, so that it can be solved in reasonable time 
for small parameters $n$ and $q$ only. However, we can use the inherent symmetry to fix some of the $x_C$ variables. I.e.\ the symmetry group acts 
transitively on the set of $a$-spaces that are not contained in $H_\infty$. For pairs of an $a$-space $A$ and a $b$-space $B$ that both are not 
contained in $H_\infty$, the different orbits under the action of the symmetry group are characterized by the invariant $\dim(A\cap B)$. 

\section{Technical details}
\label{sec_details}

In order to keep the paper more readable, we have moved some technical details, that may also be left to the reader, to this section.
The proof of Lemma~\ref{lemma_all_types_are_feasible} uses the numbers $m_i^{(j)}$ satisfying certain constraints. For completeness we state how those number 
can be computed in Algorithm \ref{algo_decomposition_packing_formula}.

\renewcommand{\algorithmicrequire}{\textbf{Input:}}
\renewcommand{\algorithmicensure}{\textbf{Output:}}
\begin{algorithm} 
\begin{algorithmic}
\REQUIRE $m_{n-1},\dots,m_1\in \N_0$ with $\sum_{i=1}^{n-1} m_i \cdot q^{i-1}=q^{n-1}$
\ENSURE  $m_i^{(j)}\in\N_0$ with $\sum_{i=1}^{n-2} m_i^{(j)} \cdot q^{i-1}=q^{n-2}$ for all $\rev{m}_{n-1}+1\le j\le q$\\ and $\sum_{j=m_{n-1}+1}^q m_i^{(j)}=m_i$ for all $1\le i\le n-2$\\[2mm]   
$h\gets n-2$
\FOR{$m_{n-1}+1 \leq j \leq q$}
  \STATE $r \gets q^{n-2}$   
  \WHILE{$r>0$}
    \STATE $t \gets \min\left\{r/q^{h-1},m_h\right\}$
    \STATE $m_h\gets m_h-t$
    \STATE $r\gets r-t\cdot q^{h-1}$
    \IF{$t=0$}
      \STATE $h\gets h-1$
    \ENDIF  
  \ENDWHILE
\ENDFOR
\RETURN $m_i^{(j)}$
\end{algorithmic}
\caption{Computing $m_i^{(j)}$}
\label{algo_decomposition_packing_formula}
\end{algorithm}

\smallskip 
\noindent 
In the three subsequent lemmas we characterize $2$-divisible sets in $\PG(3,2)$ of cardinality $s\in\{3,6,8\}$.
\begin{lemma}
  \label{lemma_2_div_card_3}
  Let $\cP$ be a $2$-divisible set of three points in $\PG(3,2)$ then $\cP$ forms a line. 
\end{lemma}
\begin{proof}
  Let $\cP=\left\{P_1,P_2,P_3\right\}$ and $L\coloneq\left\langle P_1,P_2\right\rangle$. Since all hyperplanes containing $L$ have to contain $\cP$, we have $P_3\in L$.
\end{proof}

\begin{lemma}
  \label{lemma_2_div_card_6}
  Let $\cP$ be a $2$-divisible set of six points in $\PG(3,2)$ then $\cP$ is the disjoint union of two lines.
\end{lemma}
\begin{proof}
  If $H$ is a hyperplane containing all points of $\cP$, then there is a unique point $P\le H$ with $P\notin\cP$. Since every hyperplane $H'$ that does not contain $P$ intersects $\cP$
  in cardinality $3$, so that this case cannot occur, i.e., $\cP$ is spanning. From the standard equations we compute $a_0=0$, $a_2=9$, and $a_4=6$ for the spectrum. From the MacWilliams transform 
  for the corresponding linear code we conclude the existence of a triple of points $\cP'$ forming a line. Since $\cP\backslash\cP'$ is also $2$-divisible the statement follows from 
  Lemma~\ref{lemma_2_div_card_3}.   
\end{proof}
We remark that there exists a second $2$-divisible set of six points -- a projective base of dimension $5$, which clearly cannot be embedded in $\PG(3,2)$. 

\begin{lemma}
  \label{lemma_2_div_card_8}
  Let $\cP$ be a $2$-divisible set of eight points in $\PG(3,2)$ then $\cP$ is either an affine solid or given by the points of a plane and an intersecting line without the 
  intersection point. 
\end{lemma}
\begin{proof}
  Assume that $\pi$ is a hyperplane, which is a plane in our 
  situation, containing six of the eight points and denote the unique uncovered point of $\pi$ by $P$. Each hyperplane that is incident with $P$ contains either two or six 
  of the points in $\pi$. Thus, the remaining two points form a line $L$ containing $P$. Clearly, there is a unique example up to symmetry. Otherwise each hyperplane contains 
  either $0$, $2$, or $4$ points, so that the standard equations yield that there is a unique empty hyperplane and all other hyperplanes contain exactly four points, i.e., the point 
  set is given by an affine solid.
\end{proof}
We remark that both point sets can also be described as unions of two $2$-divisible point sets, i.e., the union of two affine planes in the first case and the union of a line and 
a projective basis of size five in the second case.

\section{Tight irreducible affine vector space partitions of minimum size that can be obtained by compression}
\label{sec_compression}
In Subsection~\ref{subsec_hitting_formulas} we have shown how {\avsp}s of $\PG(n-1,2)$ can be obtained from hitting formulas by compression. In~\cite{PeitlSzeider} 
irreducible hitting formulas of minimum possible mentioning all variables where enumerated up to seven variables. Going over their list we obtain the following examples 
of tight irreducible {\avsp}s that can be obtained by compression and that have the minimum possible size $\sigma_2(n)$, see Section~\ref{sec_minimum_size}. The pairs 
of strings that can be compressed to an affine subspace are separated by horizontal lines. 

Examples for $n = 5$:
%% \begin{align*}
%% &00** & &00** \\
%% &1*0* & &1*0* \\
%% &010* & &01*0 \\
%% &1*10 & &1*10 \\
%% &*111 & &*111 \\
%% &\textcolor{blue}{0110} & &\textcolor{blue}{0101} \\
%% &\textcolor{blue}{1011} & &\textcolor{blue}{1011} \\
%% \end{align*}
\begin{align*}
&00** & &00** \\
&1*0* & &1*0* \\
&010* & &01*0 \\
&1*10 & &1*10 \\
&\underline{*111} & &\underline{*111} \\
&0110 & &0101 \\
&1011 & &1011 \\
\end{align*}

Examples for $n = 6$:
%% \begin{align*}
%% &00*** & &00*** & &00*** \\
%% &100** & &1*00* & &100** \\
%% &*100* & &1*1*0 & &01*0* \\
%% &1*1*0 & &*101* & &1*1*0 \\
%% &*1*11 & &*11*1 & &*1*11 \\
%% &\textcolor{blue}{101*1} & &\textcolor{blue}{0100*} & &\textcolor{blue}{01*10} \\
%% &\textcolor{blue}{011*0} & &\textcolor{blue}{1001*} & &\textcolor{blue}{11*01} \\
%% &\textcolor{red}{*1101} & &\textcolor{red}{011*0} & &\textcolor{red}{101*1} \\
%% &\textcolor{red}{*1010} & &\textcolor{red}{101*1} & &\textcolor{red}{110*0}
%% \end{align*}    
\begin{align*}
&00*** & &00*** & &00*** \\
&100** & &1*00* & &100** \\
&*100* & &1*1*0 & &01*0* \\
&1*1*0 & &*101* & &1*1*0 \\
&\underline{*1*11} & &\underline{*11*1} & &\underline{*1*11} \\
&101*1 & &0100* & &01*10 \\
&\underline{011*0} & &\underline{1001*} & &\underline{11*01} \\
&*1101 & &011*0 & &101*1 \\
&*1010 & &101*1 & &110*0
\end{align*}    

For $n=7$ there is a unique example:
%% \begin{align*}
%% &000*** \\
%% &10*0** \\
%% &*1**00 \\
%% &**111* \\
%% &*10**1 \\
%% &\textcolor{blue}{*0110*} \\
%% &\textcolor{blue}{*1101*} \\
%% &\textcolor{red}{*11*01} \\
%% &\textcolor{red}{*10*10} \\
%% &\textcolor{green}{0010**} \\
%% &\textcolor{green}{1001**} \\
%% \end{align*}
\begin{align*}
&000*** \\
&10*0** \\
&*1**00 \\
&**111* \\
&\underline{*10**1} \\
&*0110* \\
&\underline{*1101*} \\
&*11*01 \\
&\underline{*10*10} \\
&0010** \\
&1001** \\
\end{align*}

For $n=8$, there are $26$ irreducible hitting formulas of size $13$ mentioning all $n-1=7$ variables.
Curiously enough, compression was always successful.  Moreover, we can also obtain tight irreducible {\avsp}s of $\PG(7,2)$ 
of minimum size $\sigma_2(8)=10$ by compression starting from an irreducible hitting formulas with strictly more than $13$ terms: 
%% \begin{align*}
%% &***1**0                   & &\textcolor{blue}{00010**}& &\textcolor{blue}{1010*0*}\\
%% &***10*1                   & &\textcolor{blue}{01111**}& &\textcolor{blue}{1111*1*}\\
%% &00*0*1*                   & &\textcolor{brown}{0000*1*}& &\textcolor{brown}{0010**0}\\
%% &0*00*0*                   & &00*11**& &\textcolor{magenta}{110*1**}\\
%% &1***1*1                   & &1*100**& &111**0*\\
%% &*110**0                   & &\textcolor{brown}{1001*0*}& &\textcolor{brown}{0111**1}\\
%% &\textcolor{blue}{*010*00} & &\textcolor{magenta}{0*10**1}& &001***1\\
%% &\textcolor{blue}{*100*10} & &100**1*& &\textcolor{magenta}{000*0**}\\
%% &\textcolor{red}{0**11*1}  & &\textcolor{magenta}{1*11**0}& &\textcolor{darkgreen}{01*0***}\\
%% &\textcolor{red}{1**00*1}  & &\textcolor{red}{0*1*0*0}& &\textcolor{darkgreen}{10*1***}\\
%% &\textcolor{green}{1*00*00}& &\textcolor{red}{1*1*1*1}& &\textcolor{red}{0*011**}\\
%% &\textcolor{green}{0*10*01}& &*000*0*& &\textcolor{red}{1*000**}\\
%% &\textcolor{brown}{10*0*10}& &*10****& &1*10*1*\\
%% &\textcolor{brown}{01*0*11}& &\textcolor{green}{**101*0}& &0*11**0\\
%% &                          & &\textcolor{green}{**110*1}& &\textcolor{green}{*0001**}\\
%% &                          & &       & &\textcolor{green}{*1010**}
%% \end{align*}
\begin{align*}
&***1**0& &*10****& &1*10*1*\\
&***10*1& &*000*0*& &0*11**0\\
&00*0*1*& &100**1*& &111**0*\\
&0*00*0*& &00*11**& &\underline{001***1}\\
&1***1*1& &\underline{1*100**}& &0010**0\\
&\underline{*110**0}& &1001*0*& &\underline{0111**1}\\
&*010*00 & &\underline{0000*1*}& &110*1**\\
&\underline{*100*10} & &0*10**1& &\underline{000*0**}\\
&0**11*1& &\underline{1*11**0}& &01*0***\\
&\underline{1**00*1}  & &0*1*0*0& &\underline{10*1***}\\
&1*00*00& &\underline{1*1*1*1}& &0*011**\\
&\underline{0*10*01}& &01111**& &\underline{1*000**}\\
&10*0*10& &\underline{00010**}& &1010*0*\\
&01*0*11& &**101*0& &\underline{1111*1*}\\
& & &**110*1& &*0001**\\
&  & &       & &*1010**
\end{align*}

\bibliographystyle{alphaurl}
\bibliography{avsp_revised}{}

\end{document}